\documentclass[11pt]{amsart}

 \usepackage{amsfonts, bbm}
\usepackage{amsmath}
\usepackage{amssymb}
\usepackage{amstext}
\usepackage{amsthm}
\usepackage{tikz}
\usepackage[abs]{overpic}
\usepackage{hyperref}
\usepackage[all]{xy}
\usetikzlibrary{matrix}
\usetikzlibrary{shapes,calc,decorations.markings}
\usepackage[utf8]{inputenc}
\usepackage[T1]{fontenc}

\newcommand{\C}{\mathbb C}
\newcommand{\R}{\mathbb R}

\newcommand{\Z}{\mathbb Z}

\newcommand{\kk}{\mathbbm k}

\newcommand{\A}{\mathcal{A}}
\newcommand{\AlgL}{\mathcal{A}_{\Lambda}}
\newcommand{\M}{\mathcal{M}}
\def\d{\partial}
\newcommand{\e}{\epsilon}

\newcommand{\LCH}{\mathit{LCH}}

\newcommand{\CC}{\mathcal{C}}
\newcommand{\RR}{\mathcal{R}}
\newcommand{\id}{\operatorname{id}}
\newcommand{\tb}{\operatorname{tb}}
\newcommand{\rot}{\operatorname{rot}}
\newcommand{\sgn}{\operatorname{sgn}}
\newcommand{\im}{\operatorname{im}}
\newcommand{\Ob}{\operatorname{Ob}}
\newcommand{\Hom}{\operatorname{Hom}}
\newcommand{\Aug}{\mathcal{A}ug}
\newcommand{\Sh}{\mathcal{S}h}
\newcommand{\Rep}{\mathcal{R}ep}
\newcommand{\Mat}{\operatorname{Mat}}
\newcommand{\Ext}{\operatorname{Ext}}

\newtheorem{theorem}{Theorem}[section]
\newtheorem{lemma}[theorem]{Lemma}

\newtheorem{proposition}[theorem]{Proposition}

\theoremstyle{definition}
\newtheorem{definition}[theorem]{Definition}
\newtheorem{remark}[theorem]{Remark}
\newtheorem{example}[theorem]{Example}

\begin{document}
\author{John B.\ Etnyre}
\author{Lenhard L.\ Ng}

\address{Georgia Institute of Technology}
\email{etnyre@math.gatech.edu}

\address{Duke University}
\email{ng@math.duke.edu}

\title{Legendrian contact homology in $\mathbb{R}^3$}

\begin{abstract}
This is an introduction to Legendrian contact homology and the Chekanov--Eliashberg differential graded algebra, with a focus on the setting of Legendrian knots in $\mathbb{R}^3$. 

This is the published version of the paper, but with a section of errata added at the end.
\end{abstract}


\maketitle

\tableofcontents

\section{Introduction}

Legendrian knots have been an integral part of three dimensional contact geometry for a long time. They can be used to construct all contact manifolds from the standard contact structure on $S^3$ through surgery operations. They can be used to distinguish and understand contact structures: for example the famous tight versus overtwisted dichotomy can be expressed in terms of Legendrian knots, and contact structures on many manifolds can be distinguished using Legendrian knots. 
A fundamental problem in the theory of Legendrian knots is the classification problem: completely characterize Legendrian knots up to the natural equivalence relation, Legendrian isotopy. This is finer than the classification of smooth knots, as follows from the existence of two long-established ``classical'' invariants of Legendrian knots, the Thurston--Bennequin invariant and rotation number, which are algebro-topological numerical invariants that can distinguish between Legendrian knots of the same underlying smooth knot type. 

It was only about 20 years ago that other, ``non-classical'' invariants of Legendrian knots were developed. There are now a number of non-classical invariants. The first of these, and in many regards the most important, is \textit{Legendrian contact homology} (LCH), introduced by Chekanov \cite{Che} and Eliashberg \cite{Eli}. LCH, which is a cousin of Lagrangian intersection Floer homology, is
the homology of what has become known as the {\em Chekanov--Eliashberg differential graded algebra} (DGA), and we will sometimes abuse notation and use the terms LCH and Chekanov--Eliashberg DGA interchangeably. In the past 20 years, LCH has been shown to be a powerful invariant of Legendrian knots, but it also has revealed a beautiful internal structure and deep connections with smooth topology and symplectic geometry. 

Our goal in this paper is to present a fairly thorough overview of Legendrian contact homology, and the network of ideas radiating from it, in the setting where the theory is most fully developed: for Legendrian knots in the standard contact structure in $\R^3$.
We will discuss several points of view on the Chekanov--Eliashberg DGA and indicate the development of its properties over the years since its introduction. This discussion begins in Section~\ref{theDGA} with a description of the Chekanov--Eliashberg DGA from both combinatorial and geometric perspectives and an exploration of some basic properties of the DGA.

Trying to directly compare the Legendrian contact homology 
of two Legendrian knots is notoriously difficult (as are many noncommutative algebra problems), and as soon as the theory was developed, tools for extracting meaningful and computable information were also developed. Chief among these are augmentations of the Chekanov--Eliashberg DGA, which can be thought of as representations of LCH. 
In Section~\ref{sec:linear}, we introduce augmentations and describe how Chekanov used them to ``linearize'' Legendrian contact homology, producing an invariant that is much easier to use to distinguish between Legendrian knots than the full DGA.

Augmentations have now emerged beside LCH as an important tool in the study of Legendrian knots, in many different ways.
In one direction, simply counting augmentations over finite fields leads to surprisingly interesting invariants of Legendrian knots. 
Shortly after the introduction of the Chekanov--Eliashberg DGA, Chekanov and Pushkar defined another Legendrian invariant, namely the collection of rulings of Legendrian front diagrams.
It turns out (see Section~\ref{rulingaugs}) that the count of augmentations and the count of rulings for a Legendrian knot give the same information about a Legendrian knot. Moreover there are beautiful connections with topology: Rutherford 
discovered that the appropriate count of rulings determines a portion of the Kauffman and HOMFLY-PT polynomials of the underlying smooth knots, thus providing a subtle connection between contact geometry and smooth knot theory.

In another direction, given an augmentation, one can build on Chekanov's construction of linearized LCH to construct a more elaborate algebraic structure, which takes the form of an $A_\infty$ algebra and can be shown to be a stronger invariant than linearized LCH (see Section~\ref{ssec:a-infty}). This can further be extended to an entire $A_\infty$ category called the augmentation category, which we discuss in Section~\ref{sec:augcat}. The objects of this category are augmentations and the $A_\infty$ morphisms can be read off from the Chekanov--Eliashberg DGA, and the category imposes a rather rich structure on the set of augmentations. 
In $\R^3$ it has been proven that the augmentation category is isomorphic to a category of sheaves associated to a Legendrian knot, thus providing a connection between Legendrian knots and algebraic geometry that also touches on mirror symmetry.

Augmentations are algebraic in nature but are closely related to a geometric construction, namely Lagrangian cobordisms between Legendrian knots.
In Section~\ref{sec:fillings} we discuss how Lagrangian cobordisms induce maps between Chekanov--Eliashberg DGAs. In particular, a ``filling'' of a Legendrian knot, which is an exact Lagrangian surface bounding the knot, gives an augmentation of the DGA of the knot. Although not all augmentations arise in this fashion, one can often use augmentations as an algebraic stand-in for fillings. The augmentation category described earlier is then an algebraic analogue of a type of Fukaya category generated by fillings of a Legendrian knot.

Although one can study Legendrian contact homology on its own merits, a large amount of recent interest in the subject comes from its relation to various invariants of symplectic manifolds. In particular, there is a large class of symplectic $4$-manifolds with boundary, Weinstein domains, which can be obtained from a standard symplectic $4$-ball (or other standard pieces) by attaching Weinstein handles to Legendrian knots in the boundary. It follows from the work of Bourgeois, Ekholm, and Eliashberg that the symplectic homology of these Weinstein $4$-manifolds, as well as some invariants of their contact boundary, are essentially determined by the Chekanov--Eliashberg DGA of these Legendrian knots. This picture is still being developed but we give a brief introduction in Section~\ref{LCHandWeinstein}.

Our focus in this paper on LCH in $\R^3$ 
unfortunately forces us to omit generalizations to Legendrian knots in other contact $3$-manifolds and to higher dimensions, though we discuss these briefly in Section~\ref{extensions}. In particular, we do not consider knot contact homology, an invariant of smooth knots in $\R^3$ that is given by the Legendrian contact homology of the unit conormal bundle to the knot, which is a Legendrian $2$-torus in the $5$-dimensional unit cotangent bundle of $\R^3$. Readers interested in knot contact homology are referred to the surveys \cite{EE05,Ng06,Ng-KCH-survey,Ekholm-KCH}.

Another subject that is related to the material in this survey but beyond its scope is the rich subject of generating families, which provide another way to construct invariants of Legendrian knots.
Given a function $f:\R^n\times \R\to \R$ one can consider the plot of the ``fiberwise critical set'' $\{(t_0,\frac{\partial f}{\partial t}(x_0, t_0), f(x_0,t_0))\}$ for points $(t_0,x_0)$ such that $ \frac{\partial f}{\partial x_0}(x_0, t_0)=0$. Under some transversality conditions this set will be a Legendrian knot $\Lambda$ in the standard contact structure on $\R^3$ and we say that $f$ is a generating family for $\Lambda$. The existence of generating families for a Legendrian knot in $\R^3$ turns out to be equivalent to the existence of augmentations, by the combined work of a number of authors \cite{Fuchs, FI, FR11,  Sabloff, ChP}, and furthermore there is a natural notion of homology associated to a generating family \cite{FR11, JT06, Tr01, ST13} that turns out to be the same as linearized contact homology for the appropriate augmentation \cite{FR11}.

% ------------------------------------------------------------------------------------------------------
\vspace{11pt}
\noindent
{\bf Acknowledgments.}  
The first author was partially supported by NSF grant DMS-1608684. The second author was partially supported by NSF grants DMS-1406371 and DMS-1707652.
We thank Roman Golovko, Chindu Mohanakumar, Dan Rutherford, and Angela Wu for useful comments and error correction on earlier versions of this paper. We also thank the referee for many valuable suggestions.  
% ------------------------------------------------------------------------------------------------------

% ------------------------------------------------------------------------------------------------------
\section{Preliminaries}\label{prelims}

Throughout this paper we will focus on Legendrian knots in the standard contact $3$-manifold $(\R^3, \xi_{std})$, where 
\[
\xi_{std}= \ker (dz-y\, dx).
\]
These are knots with a regular parameterization $\gamma :\thinspace S^1\to \R^3$ such that $\gamma'(t)$ is tangent to  $\xi_{std}$ at $\gamma(t)$ for all $t \in S^1$. We will generally be interested in equivalence classes of Legendrian knots under Legendrian isotopy, which means smooth isotopy through Legendrian knots. We will assume the reader is familiar with the basics of the subject as presented in \cite{E, Geigesbook}, but recall a few ideas and notation for the reader's convenience. 

\subsection{Projections of Legendrian knots}
If $\Lambda$ is a Legendrian knot in $(\R^3, \xi_{std})$ there are two important projections to consider; see Figure~\ref{fig:leex} for examples.
\begin{figure}{\small
\begin{overpic}{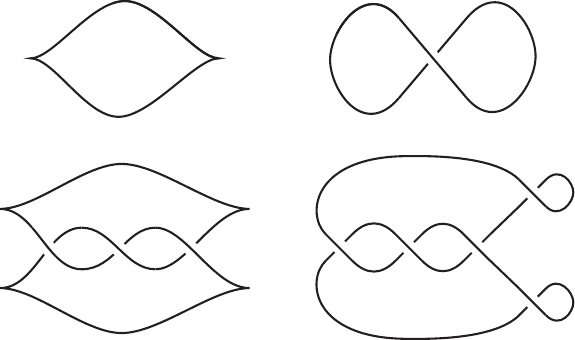}
\end{overpic}}
\caption{On the left are examples of the front projection of the unknot and the right handed trefoil knot. On the right are examples of the Lagrangian projection of the same knots; cf.\ Lemma~\ref{morsification}.}
\label{fig:leex}
\end{figure}
The first is the {\em Lagrangian projection}
\[
\Pi:\R^3\to \R^2_{xy}: (x,y,z)\mapsto (x,y).
\]
The image $\Pi(\Lambda)$ of $\Lambda$ will be an immersed curve with, generically, transverse double points. This is called the Lagrangian projection since $\Pi(\Lambda)$ is an immersed Lagrangian submanifold of the symplectic manifold $(\R^2_{xy}, d\alpha)$ (and more generally if $\Lambda$ is a Legendrian submanifold of a $1$-jet space $J^1(M) = T^*M \times \R$ then the projection $\Pi$ to $T^*M$ maps $\Lambda$ to an immersed Lagrangian in $T^*M$). Notice that $\Lambda$ is determined up to Legendrian isotopy by its Lagrangian projection. Specifically if $\Lambda$ is parameterized by $\gamma(t)=(x(t), y(t), z(t))$ then the projection $\Pi(\Lambda)$ is parameterized by the curve $t\mapsto (x(t), y(t))$ and the $z$-coordinate can be recovered from $\Pi\circ \gamma$ by 
\[
z(t)= z_0 + \int_0^t y(t) x'(t) \, dt
\]
for the appropriate choice of $z_0$, and different choices of $z_0$ give Legendrian knots isotopic to $\Lambda$. 

We will see in the next section that this projection is very useful to define the Chekanov--Eliashberg DGA of $\Lambda$, but we point out a difficulty with this projection. An immersed closed curve in $\R^2_{xy}$ only lifts to a Legendrian knot in $\R^3$ if the total integral of $y\, dx$ around the curve is zero; furthermore, even if this total integral is zero and the immersion has transverse double points, the sign of the integral of $y\,dx$ along a section of the curve from a double point back to itself will determine the over- and under-crossing information at the double point. In practice one draws Lagrangian projections of Legendrian knots modulo planar isotopy of $\R^2_{xy}$, and the crossing information determines a collection of inequalities involving the areas of the regions enclosed by the immersion. See \cite{Che} for a fuller discussion.

A consequence of these area inequalities is that not every knot diagram in $\R^2_{xy}$ (is planar isotopic to a diagram that) represents a Legendrian knot. In particular, whereas any sequence of Reidemeister moves will turn the diagram of a smooth knot into the diagram of a knot that is smoothly isotopic to the original, this is not true for Lagrangian projections of Legendrian knots and Legendrian isotopy. Nevertheless, any Legendrian isotopy can be realized by a sequence of Reidemeister moves for $\Pi(\Lambda)$, where the Reidemeister moves are restricted to double point and triple point moves (i.e., the usual Reidemeister II and III moves, but not I), along with ambient planar isotopies of an immersed curve. We will refer to these Legendrian Reidemeister moves when discussing invariance of the Chekanov--Eliashberg DGA in Section~\ref{ssec:invariance}.
See Figure~\ref{fig:Lreid}. 
\begin{figure}[htb]{\tiny
\begin{overpic}{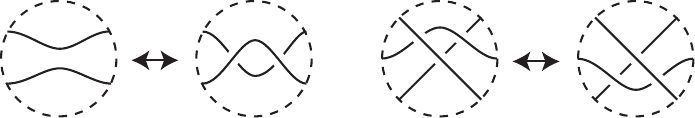}
\end{overpic}}
\caption{Reidemeister moves in the Lagrangian projection. On the left is the double point move and on the right is the triple point move. 
(These diagrams can be arbitrarily rotated or reflected.)}
\label{fig:Lreid}
\end{figure}

The {\em front projection} is the map
\[
F: \R^3\to \R^2_{xz}:(x,y,z)\mapsto (x,z). 
\]
The front projection $F(\Lambda)$ of a Legendrian knot $\Lambda$ is quite nice in that the $y$ coordinate can completely be recovered from the projection by $y=\frac{dz}{dx}$. Notice that the finiteness of the $y$ coordinate implies that no tangent lines to the projection can be vertical (parallel to the $z$-axis), and thus the front projection of $\Lambda$ cannot be immersed. Instead, the front projection contains semicubical cusps (modeled on $z^2=\pm x^3$) where the $x$ coordinate changes from increasing to decreasing or vice versa. We can also see that given a crossing in $F(\Lambda)$ one can always determine the over- and under-strand: the strand with the more negative slope will be in front of the one with the more positive slope. To see why this is the case we note that if the front projection is drawn with the $z$ axis vertical and $x$ axis horizontal, then to give $\R^3$ its standard orientation we must have that the positive $y$ axis is behind the plane of the projection and the negative axis is in front.  

The front projections of Legendrian knots are particularly easy to deal with since any diagram in $\R^2_{xz}$ meeting the above mentioned properties (no vertical tangencies, immersion away from semicubical cusps) lifts to a unique Legendrian knot.  As a consequence, it is usually easier to construct Legendrian isotopies through a sequence of moves on their front projections than through moves on their Lagrangian projections (as mentioned before, it can be tricky to check that the latter actually corresponds to an isotopy of Legendrians). There is a set of ``Legendrian Reidemeister moves'' that relate the front projections of any Legendrian isotopic knots \cite{Swiatkowski92}. 

Because Legendrian contact homology is easier to describe in the Lagrangian projection, while Legendrian isotopies are easier to see in the front projection, it will be convenient to be able to go between the two projections. This can be done through a process called Morsification or resolution (see \cite{Ng-CLI}, or \cite{E} for a brief discussion; Figure~\ref{fig:leex} illustrates two examples).
\begin{lemma}[\cite{Ng-CLI}]\label{morsification}
Given the front projection of a Legendrian knot, one can produce a diagram planar isotopic to the Lagrangian projection of a Legendrian isotopic knot by replacing the right and left cusps of the front as shown in Figure~\ref{fig:convert}.
\end{lemma}

\begin{figure}[htb]{\tiny
\begin{overpic}{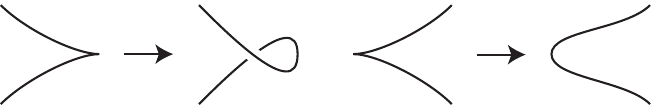}
\end{overpic}}
\caption{Resolution: changing a front projection of a Legendrian knot to a Lagrangian projection.}
\label{fig:convert}
\end{figure}

\subsection{Classical invariants of Legendrian knots}
There are three classical invariants of the Legendrian isotopy type of a Legendrian knot $\Lambda$. The first is the underlying topological knot type. The second is the framing of $\Lambda$ given to it by the contact planes. This is called the Thurston--Bennequin invariant and denoted $\tb(\Lambda)$. In the front projection this is easily computed by 
\[
\tb(\Lambda)= \text{writhe}\, (F(\Lambda)) - \# (\text{right cusps in $F(\Lambda)$)}, 
\]
where the writhe of a knot diagram is simply the number of positive crossings minus the number of negative crossings. In the Lagrangian projection $\tb(\Lambda)$ is simply the writhe of $\Pi(\Lambda)$. 

The final classical invariant of an oriented Legendrian knot $\Lambda$ is its rotation number $\rot(\Lambda)$. It is defined as a relative Euler class, but can again easily be computed in the various projections for $\Lambda$ in $(\R^3,\xi_{std})$. In the front projection 
\[
\rot(\Lambda)= \frac 12 (D - U),
\]
where $U$ and $D$ are the number of up and down cusps of $F(\Lambda)$; these are the cusps where the $z$ coordinate is increasing or decreasing, respectively, when we traverse $F(\Lambda)$ in the direction of its orientation.
In the Lagrangian projection of $\Lambda$ the rotation number is just the degree of the Gauss map for $\Pi(\Lambda)$. 

Given a Legendrian knot $\Lambda$, one can create another Legendrian knot in the same smooth knot type by one of two operations called ``stabilizations''. The {\em $\pm$-stabilization} $S_\pm(\Lambda)$ of $\Lambda$ is a Legendrian knot whose front projection is obtained from the front projection $F(\Lambda)$ by replacing a small arc of the front with a zigzag. The front of $S_\pm(\Lambda)$ has two more cusps than $F(\Lambda)$, and the two stabilizations are distinguished by orientation: for $S_+$, both of the zigzag cusps are down cusps, while for $S_-$, both are up cusps.  By the above formulas for the classical invariants, we have $\tb(S_\pm(\Lambda))=\tb(\Lambda)-1$ while $\rot(S_\pm(\Lambda))=\rot(\Lambda)\pm 1$. 

It can be shown that each of $S_\pm(\Lambda)$ is well-defined up to Legendrian isotopy, independent of the position of the zigzag. In the Lagrangian projection, $S_\pm(\Lambda)$ is obtained from $\Lambda$ by replacing a small arc of the knot diagram by a small loop with a positive crossing. An important property of Legendrian contact homology is that it vanishes for stabilizations; see Section~\ref{quant}.

% ------------------------------------------------------------------------------------------------------
\section{The Chekanov--Eliashberg DGA}\label{theDGA}
In this section we discuss the definition of the Chekanov--Eliashberg differential graded algebra of a Legendrian knot in $\R^3$. We begin with the classical definition in terms of the Lagrangian projection, followed by discussion of the geometric intuition behind the proof that it is a DGA and is invariant under Legendrian isotopy, and an alternate formulation in terms of the front projection.
We then turn to a discussion of what the Chekanov--Eliashberg DGA can and cannot tell about Legendrian knots. Finally, we consider a third definition of the Chekanov--Eliashberg DGA in terms of symplectizations that will be necessary for our later discussions, and briefly discuss extensions of the theory to other manifolds and dimensions. 

\subsection{The Chekanov--Eliashberg DGA in the Lagrangian projection}\label{firstdef}
Let $\Lambda$ be an oriented Legendrian knot in $(\R^3,\xi_{std})$. We present here the definition of the Chekanov--Eliashberg DGA $(\AlgL,\d_\Lambda)$ of $\Lambda$, or to be precise, 
the ``fully noncommutative'' version of the Chekanov--Eliashberg DGA. We first note that by a generic perturbation of $\Lambda$ through Legendrian knots we can assume the only singularities of the Lagrangian projection $\Pi(\Lambda)$ are transverse double points. 
To define the DGA, we also fix a base point $*$ on $\Lambda$ distinct from the double points.

On any contact manifold equipped with a contact $1$-form $\alpha$, there is a vector field $R_\alpha$, the \textit{Reeb vector field}, determined by $i_{R_\alpha}(d\alpha) = 0$ and $\alpha(R_\alpha)=1$; on standard contact $\R^3$, this is just the vector field $\d/\d z$. Define a \textit{Reeb chord} of $\Lambda$ to be an integral curve for the Reeb vector field with both endpoints on $\Lambda$. In our setting, the Reeb chords of $\Lambda \subset \R^3$ correspond precisely to the (finitely many) double points of $\Pi(\Lambda)$, and we label them $a_1,\ldots,a_n$.

We define $(\AlgL,\d_\Lambda)$ in stages: algebra, grading, and differential.
The algebra $\AlgL$ is the associative, noncommutative, unital algebra over $\Z$ generated by $a_1,\ldots, a_n, t, t^{-1}$, with the only relations being $t\cdot t^{-1} = t^{-1} \cdot t = 1$. We write this as
\[
\AlgL = \Z\langle a_1,\ldots,a_n,t^{\pm 1}\rangle.
\]
This is generated as a $\Z$-module by words in the (noncommuting except for $t,t^{-1}$) letters $a_1,\ldots, a_n, t, t^{-1}$, with multiplication given by concatenation; the empty word is the unit $1$. See Remark~\ref{rmk:dga-flavors} for a discussion of other versions of this algebra. 

The grading on $\AlgL$ is defined as follows. It suffices to associate a degree to each generator of $\AlgL$; then the grading of a word in the generators is the sum of the gradings of the letters in the word. The grading of $t$ is determined by the rotation number of $\Lambda$: $|t|=-2\rot(\Lambda)$ and $|t^{-1}| = 2\rot(\Lambda)$. To define the gradings of the $a_i$ we define the path $\gamma_i$ in $\R^2_{xy}$ to be the path running along $\Pi(\Lambda)$ from the overcrossing of $a_i$ to the undercrossing and missing the base point~$*$. By perturbing the diagram,
we can assume that all the strands of $\Pi(\Lambda)$ meet orthogonally at the crossings, so that the (fractional) number of counterclockwise rotations of the tangent vector to $\gamma_i$ from beginning to end, which we denote by $\rot(\gamma_i)$, will be an odd multiple of $1/4$. We then define the grading on $a_i$ to be
\[
|a_i|=2 \rot(\gamma_i) - 1/2.
\]

To define the differential on the algebra, we first decorate the Lagrangian projection of $\Lambda$. Near each crossing of $\Pi(\Lambda)$, $\R^2_{xy}$ is broken into four quadrants. We associate \textit{Reeb signs} to the quadrants as follows: we label a quadrant with a $+$ if traversing the boundary of a quadrant near $a_i$ in the counterclockwise direction one moves from an understrand to an overstrand and otherwise we label it with a $-$. See Figure~\ref{fig:signs}. We will also need an {\em orientation sign} for each quadrant. 
The orientation sign for a quadrant will be negative if it is shaded in Figure~\ref{fig:signs} and positive otherwise. 
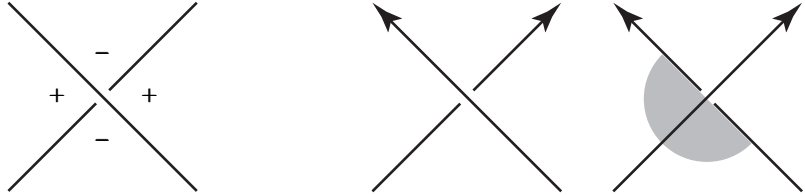
\begin{figure}[htb]{\tiny
\begin{overpic}{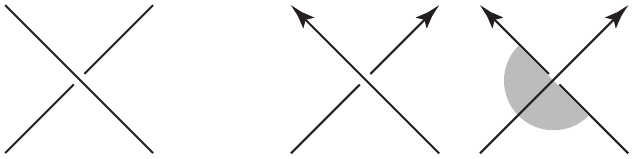}
\put(35, 20){$\pmb{-}$}
\put(18, 37){$\pmb{+}$}
\put(35, 53){$\pmb{-}$}
\put(53, 37){$\pmb{+}$}
\end{overpic}}
\caption{On the left we see the Reeb chord signs for each quadrant.  On the right we see the orientation signs, which are $-$ in the shaded quadrants and $+$ in the other quadrants. The orientation signs depend on whether the crossing is positive (right) or negative (left). 
}
\label{fig:signs}
\end{figure}

\begin{definition}\label{modulispace}
For $n \geq 0$, let $D^2_n= D^2-\{x, y_1, \ldots, y_n\}$ where $D^2$ is the closed unit disk in $\R^2$ and $x, y_1, \ldots, y_n$ are points in its boundary appearing in counterclockwise order. We call the points removed from $D^2_n$ boundary punctures. Now if $a, b_1, \ldots, b_n$ each take values in $\{a_1,\ldots, a_n\}$ then we define the set 
\[
\Delta(a;b_1,\ldots, b_n) = \{u: (D^2_n,\partial D^2_n)\to (\R^2_{xy}, \Pi(\Lambda)): \text{ satisfying (1) -- (4)}\}/\sim,
\]
where $\sim$ is reparameterization of the domain, and 
\begin{enumerate}
\item $u$ is an immersion,
\item $u$ sends the boundary punctures to the crossings of $\Pi(\Lambda)$,
\item $u$ sends $x$ to $a$ and a neighborhood of $x$ is mapped to a quadrant of $a$ labeled with a $+$ Reeb sign,
\item for $i=1,\ldots,n$, $u$ sends $y_i$ to $b_i$ and a neighborhood of $y_i$ is mapped to a quadrant of $b_i$ labeled with a $-$ Reeb sign.
\end{enumerate}
\end{definition}

Examples of such disks may be seen in Figures~\ref{fig:trefoil} and~\ref{fig:CHex}. 
One may check that if $\Delta(a;b_1,\ldots, b_n)$ is nonempty then 
\[
|a|-\sum_{i=1}^n |b_i| = 1.
\]
Given $u\in \Delta(a;b_1,\ldots, b_n)$ notice that the image of $\partial D_l^2$ is a union of $n+1$ paths $\eta_0, \ldots, \eta_n$ in $\Pi(\Lambda)$ where $\eta_0$ starts at $a$ and $\eta_i$ starts at $b_i$ (here the $\eta_i$ inherit an orientation from $D^2_n$). Let $t(\eta_i)$ be $t^k$ where $k$ is the number of times $\eta_i$ crosses the base point $*$ counted with sign according to the orientation on $\Lambda$. Associated to $u$ we have a word in $\AlgL$,
\[
w(u)= t(\eta_0)b_1 t(\eta_1)b_2\cdots b_n t(\eta_n),
\]
along with a sign, 
\[
\epsilon(u)= \epsilon(a) \prod_{i=1}^n \epsilon(b_i),
\]
where $\epsilon(c)$ for a corner $c$ is the orientation sign of the quadrant that $u$ covers at $c$. 

We can now define the differential $\d_\Lambda :\thinspace \AlgL \to \AlgL$. For $a\in \{a_1,\ldots, a_n\}$, define
\[
\partial_\Lambda(a) = \sum_{\text{\small $n\geq 0, b_1,\ldots b_n$ double points}\atop \text{\small $u\in \Delta(a; b_1,\ldots, b_n)$}} \epsilon(u) w(u).
\]
Define $\partial_\Lambda(t)=\partial_\Lambda(t^{-1})=0$ and now extend $\partial_\Lambda$ to all of $\AlgL$ by the signed Leibniz rule
\[
\partial_\Lambda (ww')= (\partial_\Lambda w)w' + (-1)^{|w|} w(\partial_\Lambda w').
\]

\begin{remark}
The fact that $\d_\Lambda(a)$ is a finite sum essentially comes from considering heights of Reeb chords. If $a$ is a double point in $\Pi(\Lambda)$, define the height $h(a)>0$ to be the difference of the $z$ coordinates of the two points on $\Lambda$ over $a$. 
\label{rmk:filtration}
If $u \in \Delta(a;b_1,\ldots,b_n)$, then by Stokes' Theorem,
\[
h(a)-\sum_{i=1}^n h(b_i)= \int_{D^2_n} u^*(dx\wedge dy) > 0.
\]
It follows that for fixed $a$, $\Delta(a;b_1,\ldots,b_n)$ can be nonempty only for finitely many choices of $b_1,\ldots,b_n$, and from there that $\d_\Lambda(a)$ is finite.
\end{remark}

This completes the definition of the Chekanov--Eliashberg DGA $(\AlgL,\d_\Lambda)$. We will state the main invariance result for this DGA in Section~\ref{ssec:invariance} below. First we make some comments about the history of versions of this DGA and present a few examples. 

\begin{remark}
The Chekanov--Eliashberg DGA was first introduced as a DGA over $\Z_2$ by Chekanov \cite{Che}; to obtain the original version from the version described above, set $t=1$ and reduce mod $2$. The DGA was subsequently lifted to a DGA over $\Z[t,t^{-1}]$ in \cite{ENS} (note that the capping paths used there are slightly different from here, but yield an isomorphic DGA). Another choice of signs was discovered in \cite{EES-ori} and the two choices were subsequently proven to give isomorphic DGAs \cite{Ng-SFT}. In the DGA over $\Z[t,t^{-1}]$, $t$ commutes with Reeb chord generators (though Reeb chords do not commute with each other), but this condition does not need to be imposed to produce a Legendrian invariant. If we stipulate that $t$ does not commute with Reeb chords, we obtain the fully noncommutative DGA presented here, which has certain advantages over the various quotients discussed in this remark that we will mention later. Some of the first appearances of the fully noncommutative DGA in the literature are in \cite{EENS,NR}. 
\label{rmk:dga-flavors}

Finally, we note that there is another version of the DGA, the ``loop space DGA'', which is more elaborate than the fully noncommutative DGA described here. This is due to Ekholm and Lekili \cite{EL}, and powers of $t$ are replaced by chains in the loop space of the Legendrian $\Lambda$. Roughly speaking, there is a relation between this loop space DGA and the usual Chekanov--Eliashberg DGA corresponding to passing to homology in the loop space. See \cite{EL} for details.
\end{remark}

\begin{remark}
To streamline the discussion, we have restricted our definition of the DGA to single-component Legendrian knots. However, this is easily extended to oriented Legendrian links in $\R^3$, with a few modifications. The main change is that we now need to choose a base point on each component, and the algebra is now $\Z\langle a_1,\ldots,a_n,t_1^{\pm 1},\ldots,t_r^{\pm 1}\rangle$ where $a_1,\ldots,a_n$ are the crossings of the link diagram and $r$ is the number of components. 
\label{rmk:links}
The differential is as usual, with the parameters $t_1,\ldots,t_r$ counting instances where disk boundaries pass through the $r$ marked points. 

One other difference from the knot case is that the grading for the Chekanov--Eliashberg DGA of a link is not well-defined, because the paths $\gamma_i$ are only defined for crossings involving a single component. 
One common way to fix a grading on the DGA of a link $\Lambda$ is to choose a Maslov potential on its front projection $F(\Lambda)$. This is a locally constant map $m :\thinspace \Lambda - (F^{-1}(\text{cusps}) \cup \{\text{base points}\}) \to \Z$ that increases by $1$ when we pass through a cusp of $F(\Lambda)$ going upwards, and decreases by $1$ when we pass through a cusp going downwards. Given a Maslov potential, we can grade the DGA associated to the front projection of $\Lambda$, see Section~\ref{frontdef} below, as follows. Generators of this DGA are crossings and right cusps of $F(\Lambda)$, along with $t_i^{\pm 1}$. We define the grading of $t_i$ to be $2\rot(\Lambda_i)$ where $\Lambda_i$ is the $i$-th component; the grading of all right cusps to be $1$; and the grading of a crossing $a$ to be $m(a_-)-m(a_+)$, where $a_-$ is the strand at $a$ with more negative slope and $a_+$ is the strand with more positive slope. See e.g.\ \cite{Ng-CLI} for a version of this approach.
\end{remark}

\begin{figure}[htb]{\small
\begin{overpic}{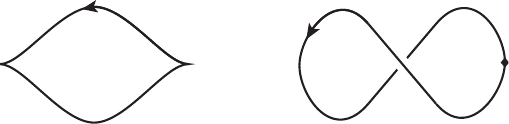}
\put(191, 39){$a$}
\put(200, 28){$\pmb{+}$}
\put(248, 28){$*$}
\end{overpic}}
\caption{The standard Legendrian unknot $\Lambda$, in the front (left) and Lagrangian (right) projections. On the right, we have added a base point, and drawn the Reeb signs at the unique double point $a$; because $a$ is a negative crossing, all orientation signs are $+$.}
\label{fig:unknot}
\end{figure}

\begin{example} \label{ex:unknot}
Let $\Lambda$ denote the Legendrian unknot shown in Figure~\ref{fig:unknot}. 
This is the ``standard Legendrian unknot'' with $\tb(\Lambda)=1$ and $\rot(\Lambda)=0$. There is one double point $a$, with grading $|a|=1$, and $\AlgL$ is generated by $a$ and $t^{\pm 1}$ with $|t|=0$. The differential $\d_\Lambda$ is completely determined by $\d_\Lambda(a)$, and this in turn has contributions from two disks corresponding to the two lobes of the figure eight. One of these does not pass through the base point $*$, while the other passes through $*$ once, opposite to the orientation of $\Lambda$. It follows that
\[
\d_\Lambda a = 1+t^{-1}.
\]
\end{example}

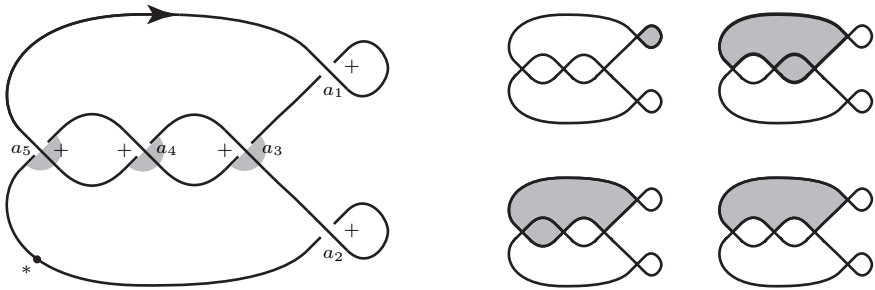
\begin{figure}[htb]{\tiny
\begin{overpic}{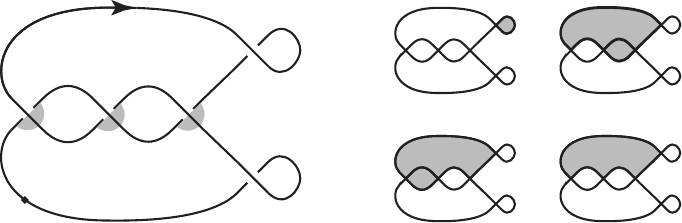}
\put(120, 11){$a_2$}
\put(120, 73){$a_1$}
\put(97, 51){$a_3$}
\put(80, 50){$+$}
\put(57, 51){$a_4$}
\put(42, 50){$+$}
\put(2, 51){$a_5$}
\put(18, 50){$+$}
\put(128, 82){$+$}
\put(128, 20){$+$}
\put(6, 5){$*$}
\end{overpic}}
\caption{The Lagrangian projection of a Legendrian trefoil knot $\Lambda$ is shown on the left. For each of the double points, the Reeb sign of one of the quadrants is shown (from which the others are easily deduced), and orientation signs are indicated by the shaded quadrants. On the right, the disks that go into the computation of $\partial_\Lambda a_1$.}
\label{fig:trefoil}
\end{figure}
\begin{example}\label{ex:trefoil}
We next consider the right handed trefoil $\Lambda$ shown in Figure~\ref{fig:trefoil}, which has $\tb(\Lambda)=1$ and $\rot(\Lambda)=0$.
The DGA is generated by the five double points labeled $a_1,\ldots, a_5$ with gradings 
\begin{gather*}
|a_1|=|a_2|=1 \\
|a_3|=|a_4|=|a_5|=0.
\end{gather*}
Figure~\ref{fig:trefoil} depicts the four disks that contribute to $\partial_\Lambda a_1$, yielding terms
(left to right, top to bottom) $1$, $a_5$, $a_3$, and $a_5a_4a_3$. One can similarly calculate the differential of $a_2$ (here $3$ of the $4$ disks pass through the marked point in a direction agreeing with the orientation of $\Lambda$, contributing a $t$ factor to the corresponding terms in $\d_\Lambda a_2$), with the conclusion that the differential $\d_\Lambda$ is given as follows:
\begin{align*}
\partial_\Lambda a_1 &= 1 +a_3+a_5+ a_5a_4a_3,\\
\partial_\Lambda a_2 &= 1-a_3t-a_5t-a_3a_4a_5t,\\
\partial_\Lambda a_3 &= \partial_\Lambda a_4=\partial_\Lambda a_5=0.
\end{align*}
\end{example}

\begin{figure}[htb]{\tiny
\begin{overpic}{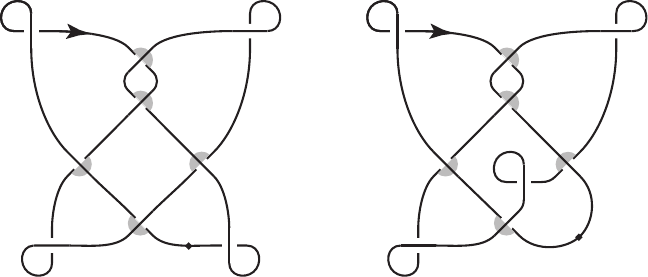}
\put(18, 123){$a_1$}
\put(111, 123){$a_2$}
\put(14, 19){$a_3$}  
\put(113, 19){$a_4$}
\put(55, 104){$a_5$}
\put(55, 82){$a_6$}
\put(26, 53){$a_7$}
\put(65, 15){$a_8$}
\put(102, 53){$a_9$}
\put(18, 111){$+$}
\put(111, 111){$+$}
\put(28, 18){$+$}
\put(103, 18){$+$}
\put(64, 97){$+$}
\put(64, 75){$+$}
\put(42, 53){$+$}
\put(65, 30){$+$}
\put(87, 53){$+$}
\put(89, 9){$*$}							
\put(195, 123){$a_1$}
\put(288, 123){$a_2$}
\put(191, 19){$a_3$}
\put(255, 38){$a_4$}
\put(232, 104){$a_5$}
\put(232, 82){$a_6$}
\put(203, 53){$a_7$}
\put(242, 15){$a_8$}
\put(279, 53){$a_9$}
\put(195, 111){$+$}
\put(288, 111){$+$}
\put(204, 18){$+$}
\put(244, 49){$+$}
\put(241, 97){$+$}
\put(241, 75){$+$}
\put(219, 53){$+$}
\put(242, 30){$+$}
\put(264, 53){$+$}
\put(281, 14){$*$}
\end{overpic}}
\caption{The Lagrangian projection of the two Chekanov knots. On the left is $\Lambda_1$ and on the right is $\Lambda_2$. For each of the double points, the Reeb sign of one of the quadrants is shown, and quadrants with negative orientation signs are shaded.}
\label{fig:CHex}
\end{figure}
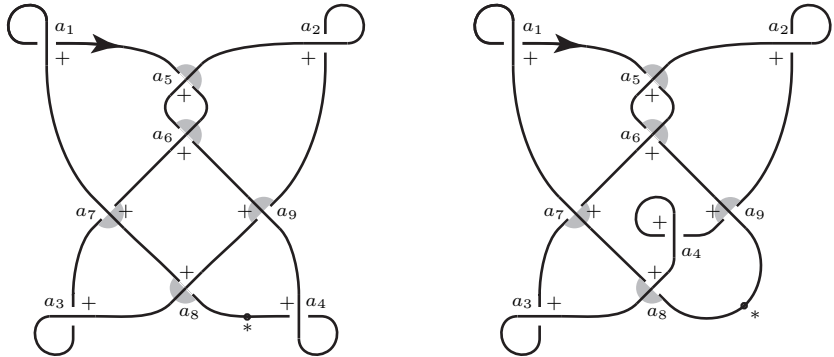
\begin{example}\label{ex:chex}
Here we consider the Chekanov $m(5_2)$ knots, a famous pair of Legendrian knots that were the first examples of Legendrian knots with the same classical invariants to be proved to be distinct \cite{Che}.
These are shown in Figure~\ref{fig:CHex}; they are both of topological type $m(5_2)$ (the mirror of $5_2$), and it is easy to check that they both have $\tb=1$ and $\rot=0$.
Each knot diagram has nine crossings. 
The gradings for the crossings of $\Lambda_1$ are
\begin{gather*}
|a_1|=|a_2|=|a_3|=|a_4|=1, \\
|a_5|= 2, \\ |a_6|=-2,\\
|a_7|=|a_8|=|a_9|=0
\end{gather*}
and the differential is
\begin{align*}
\partial_{\Lambda_1} a_1 &= 1 + a_7 + a_7a_6a_5,\\
\partial_{\Lambda_1} a_2 &= 1- a_9 - a_5a_6a_9,\\
\partial_{\Lambda_1} a_3 &= 1+a_8a_7,\\
\partial_{\Lambda_1} a_4 &= 1 + a_9a_8t^{-1},\\
\partial_{\Lambda_1} a_ i &=0, \quad i\geq 5.
\end{align*}

\begin{figure}[htb]{\tiny
\begin{overpic}{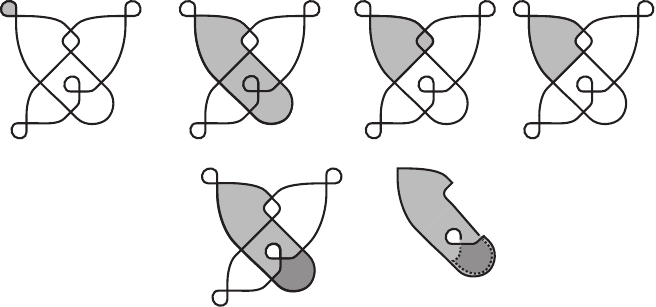}
\end{overpic}}
\caption{The five disks that go into the computation of $\partial_\Lambda a_1$ for the Chekanov example knot $\Lambda_2$ from Figure~\ref{fig:CHex}. The disk on the left of the bottom row is immersed, and
the darker shaded part indicates where the immersion is two-to-one. The final picture on the bottom row is another view of this immersed disk: the boundary of the disk is slightly offset where the immersion is two-to-one to better indicate the overlapping region.}
\label{fig:diskCH}
\end{figure}

For $\Lambda_2$, the gradings are
\begin{gather*}
|a_1|=|a_2|=|a_3|=|a_4|=1, \\
|a_5|=  |a_6|= |a_7|=|a_8|=|a_9|=0
\end{gather*}
(for future reference, note the lack of crossings of degree $\pm 2$, cf.\ $\Lambda_1$). The differential for $\Lambda_2$ is a bit trickier to visualize than in the previous examples because one of the immersed disks is not embedded. Specifically, Figure~\ref{fig:diskCH} shows the $5$ disks that contribute to $\partial_{\Lambda_2} a_1$, and the last of these is not embedded. The full differential is:
\begin{align*}
\partial_{\Lambda_2} a_1 &= 1 + a_5+a_7 + a_7a_6a_5+ t^{-1}a_9a_8t^{-1} a_5,\\
\partial_{\Lambda_2} a_2 &= 1- a_9 - a_5a_6a_9,\\
\partial_{\Lambda_2} a_3 &= 1+a_8a_7,\\
\partial_{\Lambda_2} a_4 &= 1 + a_8t^{-1}a_9,\\
\partial_{\Lambda_2} a_ i &=0, \quad i\geq 5.
\end{align*}

\end{example}

\subsection{$\d^2=0$ and invariance}
\label{ssec:invariance}

We now state the two basic properties of the Chekanov--Eliashberg DGA $(\AlgL,\d_\Lambda)$, which can be summarized as ``$\d^2=0$'' and ``invariance''. Versions of these results were proved combinatorially in \cite{Che,ENS} and analytically in \cite{EES}.

\begin{theorem}\label{thm:d2}
Given an oriented Legendrian knot $\Lambda$ in $(\R^3,\xi_{std})$ and a base point $*\in \Lambda$, we have that $\partial_\Lambda$ lowers degree by $1$ and $\partial_\Lambda \circ \partial_\Lambda =0$. Thus $(\AlgL, \partial_\Lambda)$ has the structure of a differential graded algebra with gradings taking values in $\Z$.
\end{theorem}

\begin{remark}
All of the examples given in Section~\ref{firstdef} trivially satisfy $\d_\Lambda^2=0$. An example where this nontrivially holds is given in Appendix~\ref{app:unknotDGA}; for a simpler example, see the figure eight knot in \cite[Example~4.17]{E}.
\end{remark}

If we change $\Lambda$ by Legendrian isotopy, the DGA $(\AlgL,\d_\Lambda)$ changes in a prescribed way called \textit{stable tame isomorphism}, a somewhat involved notion due to Chekanov that we now define.
First, an {\em elementary automorphism} of a DGA $(\Z\langle a_1,\ldots, a_n, t^{\pm 1}\rangle, \partial)$ is a chain map $\phi:\thinspace \Z\langle a_1,\ldots, a_n, t^{\pm 1}\rangle\to\Z\langle a_1,\ldots, a_n, t^{\pm 1}\rangle$ for which there is some $1\leq j\leq n$ such that the map has the following form:
\begin{align*}
\phi(a_j) &= \pm t^k a_j t^{\ell} + u, &\quad u&\in \Z\langle a_1,\ldots,\widehat{a}_j,\ldots,a_n,t^{\pm 1}\rangle, ~k,\ell\in\Z \\
\phi(a_i) &= a_i, & i&\neq j \\
\phi(t) &= t.
\end{align*}
Note that elementary automorphisms are in particular invertible. A {\em tame isomorphism} between two DGAs $(\Z\langle a_1,\ldots, a_n, t^{\pm 1}\rangle, \partial)$ and $(\Z\langle a'_1,\ldots, a'_n, t^{\pm 1}\rangle, \partial)$ is a chain map given by a composition of some number of elementary automorphisms of $(\Z\langle a_1,\ldots, a_n, t^{\pm 1}\rangle, \partial)$ and the algebra map sending $t \mapsto t$ and $a_i \mapsto a_i'$ for all $i$.

The {\em grading $k$ stabilization} of the DGA $(\Z\langle a_1,\ldots, a_n, t^{\pm 1}\rangle, \partial)$ is the algebra $\Z\langle e_k, e_{k-1}, a_1,\ldots, a_n, t^{\pm 1}\rangle$ where $|e_k|=k$ and $|e_{k-1}|=k-1$, equipped with the differential $\partial$ agreeing with the original differential $\partial$ on the $a_i$ and satisfying $\d(e_k) = e_{k-1}$, $\d(e_{k-1})=0$.

Finally, two DGAs are {\em stable tame isomorphic} if after each is stabilized some number of times, they become tame isomorphic. We can now state the invariance result.

\begin{theorem}\label{maininvariance}
 The stable tame isomorphism type of $(\AlgL, \partial_\Lambda)$ is an invariant of $\Lambda$ under Legendrian isotopy and choice of base point. 
\end{theorem}

One may readily check (see e.g.\ \cite{ENS}) that stable tame isomorphism is a special case of chain homotopy equivalence and thus quasi-isomorphism. (See Remark~\ref{rmk:trivial} for an example where quasi-isomorphism does not imply stable tame isomorphism.) It follows that the homology $H_*(\AlgL,\d_\Lambda)$, the \textit{Legendrian contact homology} of $\Lambda$, is invariant under Legendrian isotopy.

We now provide a sketch of the proofs of the $\d^2=0$ (Theorem~\ref{thm:d2}) and invariance (Theorem~\ref{maininvariance}) results. We begin with invariance. There is a combinatorial proof of invariance, originally due to Chekanov, that checks 
that if the Lagrangian projection $\Pi(\Lambda)$ undergoes ambient isotopy in $\R^2$, a double point move, or a triple point move (see Figure~\ref{fig:Lreid} and the discussion around it), then $(\AlgL,\partial_\Lambda)$ changes by a stable tame isomorphism. Clearly ambient isotopy does not change any relevant data in the definition of $(\AlgL,\partial_\Lambda)$. It turns out there are several triple points moves one must consider depending on the Reeb sign of the quadrants one sees in the local picture of the move, Figure~\ref{fig:Lreid}. One may check that in each case the DGA is unchanged or undergoes a tame isomorphism. For a double point move one may also check that the algebra undergoes a stabilization followed by a tame isomorphism.  See \cite{Che, ENS}. We remark that there is also a more geometric proof of Theorem~\ref{maininvariance} that closely resembles a standard bifurcation argument for invariance of Floer homology, see \cite{EES}.

The proof of $\d^2=0$ in Theorem~\ref{thm:d2} is a fairly standard ``Morse--Floer'' type argument that is less technical than invariance, and we discuss it more fully here. 
Recall $\partial_\Lambda$ is computed by computing ``rigid'' (i.e., appearing in $0$-dimensional moduli spaces) immersions of a disk with boundary on $\Pi(\Lambda)$. We will see below that if one considers (the closure of) a $1$-dimensional space of immersed disks, then in their boundaries one sees terms contributing to $\partial_\Lambda^2$, and indeed all such terms are in the boundary of some $1$-dimensional space of disks. Thus since the signed count of the points in the boundary of an oriented $1$-dimensional manifold is $0$, it follows that $\partial_\Lambda^2=0$.

To give some details on $\d^2=0$, suppose we consider the space 
\[
\widehat{\Delta}(a;b_1,\ldots, b_n) = \{u: (D^2_n,\partial D^2_n)\to (\R^2_{xy}, \Pi(\Lambda)): \text{ satisfying (1) -- (4)}\}/\sim,
\]
where $\sim$ is reparameterization, and 
\begin{enumerate}
\item $u$ is an immersion on the interior of $D^2_n$ and has a finite number of branched points on $\partial D^2_n$,
\item $u$ sends the boundary punctures to the crossings of $\Pi(\Lambda)$,
\item $u$ sends $x$ to $a$ and a neighborhood of $x$ is mapped to a quadrant of $a$ labeled with a $+$ or covers three quadrants with two labeled with a $+$,
\item $u$ sends $y_i$ to $b_i$ and the image of a neighborhood of $y_i$ either covers one quadrant at $b_i$ labeled with a $-$, or covers three quadrants with two quadrants labeled with a $-$.
\end{enumerate}
Notice that this is the same space as ${\Delta}(a;b_1,\ldots, b_n)$ except we now allow disks with locally non-convex corners and branched points along the boundary. See Figure~\ref{fig:dsquared}.

\begin{figure}[htb]{\tiny
\begin{overpic}{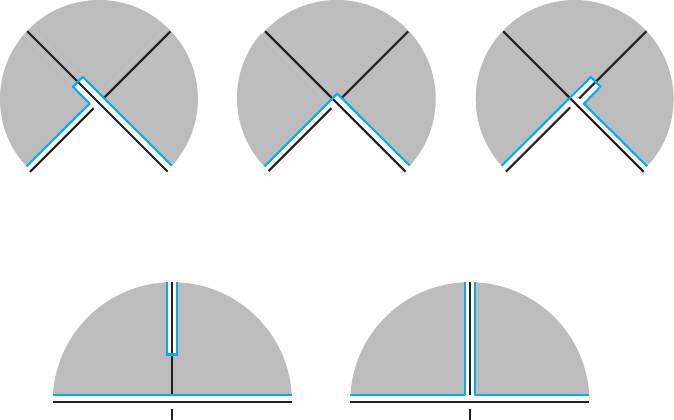}
\end{overpic}}
\caption{The new types of disks in $\widehat{\Delta}(a;b_1,\ldots, b_n)$. Along the top row we see a disk with a branch point on the right and left, and in the center we see a non-convex corner; the result is that a disk with such a non-convex corner is in the interior of a moduli space of dimension larger than 0. On the bottom row we see a disk with a branch point moving toward the boundary of the disk, at which point it limits to the union of two disks.}
\label{fig:dsquared}
\end{figure}

As with ${\Delta}(a;b_1,\ldots, b_n)$, the dimension of $\widehat{\Delta}(a;b_1,\ldots, b_n) $ is given by 
\[
\left(|a|-\sum_{i=1}^n |b_i|\right) -1,
\]
see \cite{ENS}.
One may also check that the dimension of $\widehat{\Delta}(a;b_1,\ldots, b_n)$ is simply the number of branch points plus the number of non-convex corners. It is easy to see that the branch point can slide along $\Pi(\Lambda)$ and hence such a disk will be in a family of disks with a degree of freedom coming from the branch point. Moreover, as shown in Figure~\ref{fig:dsquared}, a non-convex corner is part of a family of disks with branch points.

We now notice that if a sequence of disks has a branch point that approaches an edge of the disk, as shown in the bottom row of Figure~\ref{fig:dsquared}, then it will limit to the union of the image of two disks, each of which has fewer branch points that the disks in the original sequence. We call the union of these two disks a broken disk. So if $\widehat{\Delta}(a;b_1,\ldots, b_n)$ is one dimensional then we can compactify it by adding broken disks. With a little thought one can see that any term in $\partial_\Lambda^2 a$ is a broken disk that is in the boundary of some $1$-dimensional $\widehat{\Delta}(a;b_1,\ldots, b_n)$. The boundary components cancel in pairs in $\partial_\Lambda^2 a$ and it follows that  $\partial_\Lambda^2 a=0$.

\begin{remark}\label{whyalg}
Those familiar with standard Floer theory for pairs of embedded Lagrangian submanifolds might expect that instead of an algebra we could just define our chain complex to be a vector space generated by the double points in $\Pi(\Lambda)$, with the differential counting immersed disks with one positive and one negative puncture. However, we are forced to consider the full algebra because the two cancelling ends of a $1$-dimensional moduli space may have different combinatorics. As an example, see Figure~\ref{fig:dsquaredex}. The figure on the left consists of a broken disk where each of the two disks has one positive and one negative corner, as in Lagrangian Floer theory. It is however cancelled by the figure on the right, which is a broken disk where one disk has two negative corners and the other has none. This illustrates the need for disks with arbitrary numbers of negative corners to ensure $\partial_\Lambda^2=0$.

\begin{figure}[ht]{\tiny
\begin{overpic}{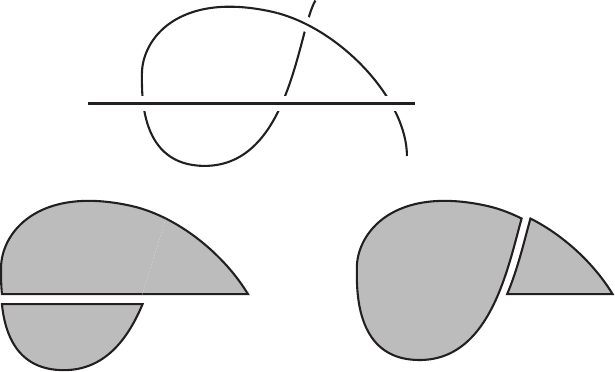}
\put(120, 120){$+$}
\put(72, 133){$+$}
\put(135, 162){$+$}
\put(190, 133){$+$}
\put(127, 135){$a$}
\put(60, 122){$b$}
\put(150, 156){$c$}
\put(173, 133){$d$}
\put(30, 15){$u_1$}
\put(45, 55){$u_2$}
\put(264, 48){$\widetilde u_1$}
\put(204, 45){$\widetilde u_2$}
\end{overpic}}
\caption{The top diagram is a portion of some Lagrangian projection $\Pi(\Lambda)$. On the bottom are disks contributing to $\partial_\Lambda a$. On the left we see $u_1$ contributes $b$ to $\partial_\Lambda a$, while $u_2$ contributes $d$ to $\partial_\Lambda^2 a$. On the left $\widetilde u_1$ contributes $dc$ to $\partial_\Lambda a$ while $\widetilde u_2$ shows the differential of $dc$ has a term $d$ in it. The two resulting $d$ terms in $\partial_\Lambda^2 a$ cancel.}
\label{fig:dsquaredex}
\end{figure}
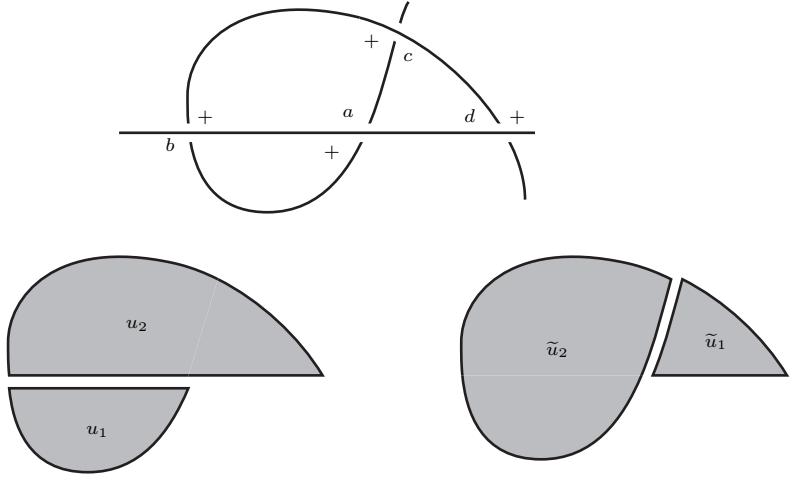
One could then ask why we can restrict to disks with exactly one positive corner. The essential reason is that by Stokes' Theorem, there are no disks with boundary on $\Pi(\Lambda)$ with all convex corners where all of the corners are negative, and so any broken disk in the compactification of $\widehat{\Delta}(a;b_1,\ldots, b_n)$ must be a union of two disks, each of which has one positive corner. The general framework of Symplectic Field Theory \cite{EGH} suggests that we could expand our disk count to include disks with multiple positive corners, and indeed this can be done; see \cite{EkholmSFT,Ng-SFT}. From this viewpoint, we can filter by the number of positive corners, and LCH is a filtered quotient of a larger SFT invariant.
\end{remark}

\subsection{The Chekanov--Eliashberg DGA in the front projection}\label{frontdef}
While the Lagrangian projection is where the Chekanov--Eliashberg DGA is naturally defined (cf.\ the geometric definition in Section~\ref{ssec:symplectization} below), and where it is easiest to prove $\d^2=0$ and invariance, it is frequently helpful to have a version of $(\AlgL,\partial_\Lambda)$ in terms of the front projection of $\Lambda$.
With the aid of Lemma~\ref{morsification}, which converts front diagrams to Lagrangian diagrams, this is a simple task (see \cite{Ng-CLI} for more details). Specifically, given a generic front projection $F(\Lambda)$ of an oriented Legendrian knot $\Lambda$ and a base point $*$ away from right cusps and double points, the algebra $\AlgL$ is generated over $\Z$ by formal symbols $t$ and $t^{-1}$ and the set $\{a_1,\ldots, a_n\}$ of double points and right cusps in the diagram. The grading of the cusps are always $1$ and the gradings of a crossing $a$ is again computed using a path $\gamma$ in $F(\Lambda)$ from the overcrossing of $a$ to the undercrossing of $a$ that misses the marked point $*$. Given $\gamma$ we have 
$|a| = D(\gamma) - U(\gamma)$,
where $D(\gamma)$ and $U(\gamma)$ are the number of downward and upward cusps one encounters 
while traversing $\gamma$.
To compute $\partial_\Lambda$ we consider maps of the unit disk $D_n^2$ with $(n+1)$ boundary punctures $x, y_1,\ldots, y_n$, $u:D^2_m\to \R^2_{xz}$, that for generators $a, b_1, \ldots, b_n$ satisfy
\begin{enumerate}
\item $u$ is an immersion on the interior of $D^2$,
\item $u$ along the boundary of $\partial D^2_n$ is an immersion except at cusps where the image of $u$ is as shown in Figure~\ref{fig:fc},
\item $u$ sends each boundary puncture to a crossing or right cusp of $F(\Lambda)$,
\item $u$ sends $x$ to $a$, and a neighborhood of $x$ is mapped to a (leftward-facing) quadrant of $a$ labeled with a $+$ Reeb sign if $a$ is a crossing, or to the leftward-facing region bounded by the cusp if $a$ is a right cusp,
\item $u$ sends $y_i$ to $b_i$, and a neighborhood of $y_i$ is mapped to a quadrant of $b_i$ labeled with a $-$ Reeb sign if $b_i$ is a crossing, or to one of the diagrams in the top row of Figure~\ref{fig:fcontribution} if $b_i$ is a cusp.
\end{enumerate}
\begin{figure}[htb]{\tiny
\begin{overpic}{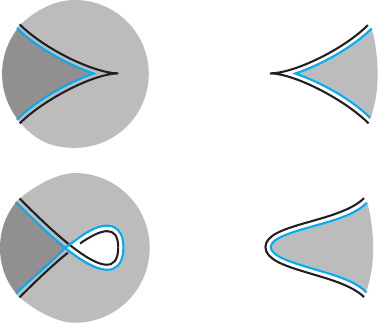}
\end{overpic}}
\caption{Top row are cusps in the front projection and the local image of the immersion $u$ near the cusp point (darker shading indicates the map is locally two to one). The bottom row is the image of a corresponding immersion in the Lagrangian projection. The image of the boundary of the disk is slightly offset for the sake of visibility.  }
\label{fig:fc}
\end{figure}
The contribution of $u$ to $\partial_\Lambda a$ is
\[
w(u)= t(\eta_0)c(b_1) t(\eta_1)c(b_2)\cdots c(b_n) t(\eta_n)
\]
where the $\eta_i$ are the images of the arcs in $\partial D^2_n$ and $t(\eta_i)$ are the powers of $t$ as defined in the original definition of the differential in Section~\ref{theDGA}, and $c(b_i)=b_i$ unless $b_i$ is a right cusp and the image of $u$ near $b_i$ looks like the rightmost diagram in Figure~\ref{fig:fcontribution}, in which case $c(b_i) = b_i^2$.
\begin{figure}[htb]{\tiny
\begin{overpic}{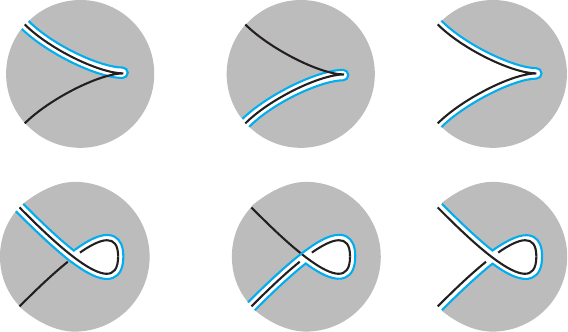}
\end{overpic}}
\caption{The top row shows the local picture of the image of $u$ near a right cusp $b$. The bottom row shows the corresponding immersion in the Lagrangian projection. In the left and middle figures the contribution $c(b)$ is $b$ while in the right figure the contribution is $b^2$. The image of the boundary of the disk is slightly offset for the sake of visibility. }
\label{fig:fcontribution}
\end{figure}
Now the differential is 
\[
\partial_\Lambda a= \begin{cases} \sum \epsilon(u) w(u) & a \text{ is a crossing}\\
1+ \sum \epsilon(u) w(u) & a \text{ is a right cusp}
\end{cases}
\]
where the sum is taken over all disks $u$, up to reparameterization, described above, and $\epsilon(u)$ is $\pm 1$ depending on whether the number of $-$ corners in $u$ that cover a downward-facing (bottommost) quadrant is even or odd.
(This choice of signs differs slightly from the orientation signs for the resolution of the front as shown in Figure~\ref{fig:signs}, but is equivalent via an automorphism that negates some of the generators of the DGA, and is slightly more convenient for computations.)
See Section~\ref{quant} and Appendix~\ref{app:unknotDGA} for computed examples of the DGA in the front projection.

\subsection{Some observations about the Chekanov--Eliashberg DGA}\label{quant}
Here we qualitatively discuss what the Chekanov--Eliashberg DGA can and cannot detect about a Legendrian knot.

\bigskip
\noindent
{\bf Vanishing of the DGA.}
We begin with a simple observation: any $\pm$-stabilization of a Legendrian knot has vanishing contact homology.
\begin{proposition}[\cite{Che}]
If $\Lambda$ is a stabilized Legendrian knot 
\label{prop:trivial}
then the Legendrian contact homology of $\Lambda$ is trivial. 
\end{proposition}
\begin{proof}
When stabilizing a knot we add a small loop to the Lagrangian projection of the knot. The new double point $a$ can be chosen to have small height (see Remark~\ref{rmk:filtration}), so that $h(a)$ is smaller than $h(b)$ for any other double point. Then by the Stokes' Theorem argument from Remark~\ref{rmk:filtration}, the only contribution to $\d_\Lambda a$ comes from the disk bounded by the loop; that is, $\d_\Lambda a = 1$. 
Now if $h$ is any element in the kernel of $\partial_\Lambda$ then $\partial_\Lambda (ah)=h$, so every cycle is a boundary. 
\end{proof}

\begin{remark}
\label{rmk:trivial}
The DGA of a stabilized knot provides a negative answer to the question: if two Chekanov--Eliashberg DGAs have isomorphic homology, are they necessarily stable tame isomorphic? Indeed, define the Euler characteristic of a DGA to be the difference between the numbers of even-graded generators and odd-graded generators (for the DGA of a Legendrian knot, this is just the Thurston--Bennequin number). It is clear that Euler characteristic is invariant under stable tame isomorphism, while any two stabilized knots have quasi-isomorphic DGAs even if they have different $\tb$. 

There is one case where quasi-isomorphism implies stable tame isomorphism. If two Chekanov--Eliashberg DGAs have vanishing homology and the same Euler characteristic, then they are stable tame isomorphic. To see this, start with a DGA $(\A,\d)$ with vanishing homology, so that $\partial(x) = 1$ for some $x\in\A$. Label the Reeb chord generators of $\A$ as $a_1,\ldots,a_n$ in decreasing order of height, so that $\d(a_i)$ does not involve $a_1,\ldots,a_i$. Stabilize by adding two generators $a_0,b$ of degree $2,1$ respectively, with $\d(a_0)=b$, $\d(b)=0$, and let $(\A',\d)$ denote the result. Apply the elementary automorphism $\phi$ of $\A'$ that sends $b$ to $b-x$; the new differential $\d'=\phi\d\phi^{-1}$ on $\A'$ satisfies $\d'(a_0) = b-x$, $\d'(b)=1$. Now successively conjugate $\d'$ by the automorphism sending $a_i$ to $a_i+b\d'(a_i)$ for $i=0,\ldots,n$. The resulting differential $\d''$ is given by $\d''(b)=1$ and $\d''(a_i) = 0$ for all $i=0,\ldots,n$, and it is then easy to check that the stable tame isomorphism type of $(\A',\d'')$ is determined by its Euler characteristic.

\end{remark}

Proposition~\ref{prop:trivial} brings up the interesting question of whether or not vanishing of the LCH of a Legendrian knot implies that the knot is stabilized. This was an open question for some time, but was finally answered negatively by Sivek in \cite{Sivek}, using the Legendrian knot on the left hand side of Figure~\ref{fig:Sivek}. This knot is of topological type $m(10_{132})$ and is non-destabilizable because it has maximal $\tb$, as calculated in \cite{ng-arc}. On the other hand, the LCH of this knot vanishes. Indeed, if we label Reeb chords as in Figure~\ref{fig:Sivek} and choose a base point say near the bottom right cusp, the differential satisfies:
\begin{align*}
\d(a_1) &= 1+a_8+a_8a_4a_3 \\
\d(a_2) &= 1+a_5a_7 \\
\d(a_6) &= -a_7a_8 \\
\d(a_3) &= \d(a_4) = \d(a_5) = \d(a_7) = \d(a_8) = 0;
\end{align*}
it follows that $\d(a_2a_8+a_5a_6) = a_8$ and so $\d(a_1-(a_2a_8+a_5a_6)(1+a_4a_3)) = 1$.
It is also interesting to note that Sivek also produced another Legendrian knot in this knot type that has non-vanishing LCH; see Figure~\ref{fig:Sivek} again.
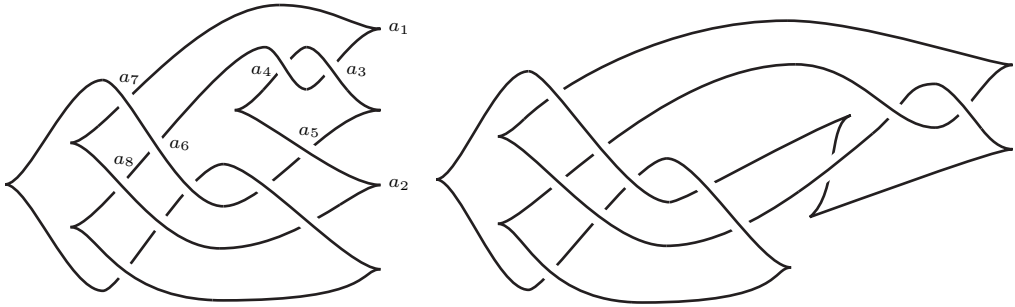
\begin{figure}[ht]{\tiny
\begin{overpic}{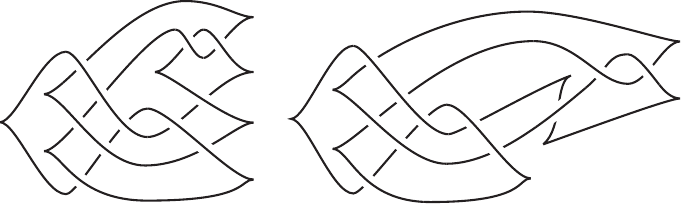}
\put(125, 92){$a_1$}
\put(125, 40.25){$a_2$}
\put(110, 78){$a_3$}
\put(77, 78){$a_4$}
\put(93, 58){$a_5$}
\put(53, 53){$a_6$}
\put(35, 77){$a_7$}
\put(34, 49){$a_8$}
\end{overpic}}
\caption{Two Legendrian representatives of the knot $m(10_{132})$ with maximal Thurston--Bennequin invariant. The one of the left has trivial contact homology and the one on the right has non-trivial contact homology. }
\label{fig:Sivek}
\end{figure}

\bigskip
\noindent
{\bf The DGA of the unknot.}
It is also interesting to note, as first observed in \cite{CNS}, that the Chekanov--Eliashberg DGA does not characterize the standard Legendrian unknot. 

\begin{proposition}[cf.\ \cite{CNS}]
For $m \geq 1$, the Legendrian knot shown in Figure~\ref{fig:unknotDGA}, which is topologically the pretzel knot $P(3,-3,-3-m)$, has a DGA that is stable tame isomorphic to the DGA of the standard Legendrian unknot.
\label{prop:unknotDGA}
\end{proposition}

\noindent
The proof of Proposition~\ref{prop:unknotDGA} was omitted in \cite{CNS} (see also Remark~\ref{rmk:unknotDGA} below); however, in Appendix~\ref{app:unknotDGA}, we provide an explicit stable tame isomorphism in the case $m=1$, which can be readily extended to general $m$.

\begin{figure}[ht]{\small
\begin{overpic}{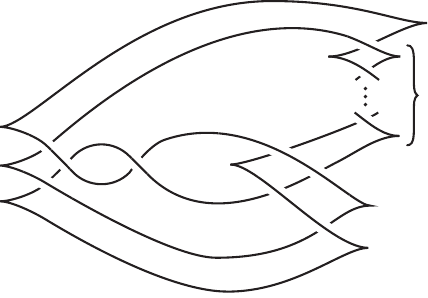}
\put(204,96){$m$}
\end{overpic}}
\caption{Legendrian representative of the pretzel knot $P(3,-3, -3-m)$ whose DGA is stable tame isomorphic to the DGA of the standard unknot.}
\label{fig:unknotDGA}
\end{figure}

\begin{remark}
\label{rmk:unknotDGA}
The family of Legendrian knots in Figure~\ref{fig:unknotDGA} is actually slightly different from the family given in \cite{CNS}. For $m \geq 2$, both families satisfy the statement of Proposition~\ref{prop:unknotDGA}. For $m=1$, which corresponds to the topological knot $m(10_{140})$, the atlas \cite{CN} depicts two Legendrian representatives, which we denote here for concreteness by $\Lambda_1$ and $\Lambda_2$ in the order given in the atlas. The knot shown in Figure~\ref{fig:unknotDGA} (for $m=1$) is $\Lambda_1$, while the knot given in \cite[\S 4.3]{CNS} is $\Lambda_2$. Computations with Gr\"obner bases suggest that the DGA for $\Lambda_2$, unlike for $\Lambda_1$, may in fact not be the same as the DGA for the unknot. This does not affect the results of \cite{CNS} except that the $m(10_{140})$ diagram given there should be replaced by the one given in Figure~\ref{fig:unknotDGA}.
\end{remark}

It can be shown that given any Legendrian knot $\Lambda$, one can produce arbitrarily many distinct Legendrian knots whose DGA is stable tame isomorphic to the DGA for $\Lambda$, by taking the connected sum of $\Lambda$ with any number of disjoint copies of the $P(3,-3,-3-m)$ knots shown in Figure~\ref{fig:unknotDGA}. See \cite{E} for the definition of connected sum for Legendrian knots.

\bigskip
\noindent
{\bf Distinguishing arbitrarily many Legendrian knots.}
The previous two observations indicated the limits of the Chekanov--Eliashberg DGA, but we now observe that the DGA can distinguish arbitrarily many Legendrian knots of a single topological type with the same $\tb$ and $\rot$. Specifically, consider a Legendrian twist knot as shown in Figure~\ref{fig:chekanovfamilies}, where the box contains $m$ half-twists each represented by a $Z$ or $S$, for even $m \geq 2$. For fixed $m$, this gives a family of Legendrian knots of the same topological type and all with $(\tb,\rot)=(1,0)$. We will see in Section~\ref{subsec:linear} that linearized contact homology, which is derived from the DGA, recovers the unordered pair $\{k,l\}$, where $k$ and $l$ are the number of $Z$'s and $S$'s in the box, with $k+l=m$. It follows that there are at least $\frac{m}{2}+1$ distinct Legendrian knots representing a single topological twist knot, all with $(\tb,\rot)=(1,0)$. This was first proven in \cite{EFM}, building on work of Eliashberg. For $m=2$, the knots represented in the box by $ZS$ and $SS$ turn out to be the Chekanov $m(5_2)$ knots $\Lambda_1$ and $\Lambda_2$ respectively from Figure~\ref{fig:CHex}.

\begin{figure}[ht]{
\begin{overpic}{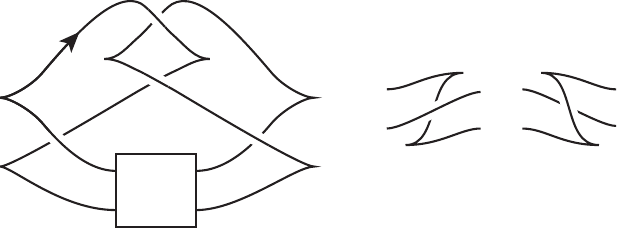}
\put(71, 15){$m$}
\put(202, 25){$Z$}
\put(270, 25){$S$}
\end{overpic}}
\caption{A Legendrian twist knot. The box is replaced with a tangle formed by concatenating $m$ of the $Z$ and $S$ tangles shown on the right, in any order.}
\label{fig:chekanovfamilies}
\end{figure}

\begin{remark}
In fact for fixed $m$, there are exactly $\lceil \frac{m^2}{8} \rceil$ isotopy classes of Legendrian twist knots of the relevant topological type with $(\tb,\rot) = (1,0)$. This is proven in \cite{ENV} using a combination of linearized contact homology, knot Floer homology, and convex surface theory.
\end{remark}

\subsection{The Chekanov--Eliashberg DGA in the symplectization}
\label{ssec:symplectization}
In this section we discuss an alternate way to define the Chekanov--Eliashberg DGA using the symplectization of $(\R^3,\xi_{std})$. This definition is much more in the spirit of Symplectic Field Theory as set up by Eliashberg, Givental, and Hofer \cite{EGH}.
It also has the advantage of allowing one to consider Lagrangian cobordisms between Legendrian knots, as we will do in Section~\ref{sec:fillings}.

As usual we start with a Legendrian knot $\Lambda$ in $(\R^3,\xi_{std})$ with a marked point $*\in \Lambda$. The symplectization of $(\R^3,\xi_{std})$ is the 
symplectic manifold 
\[
(\R\times \R^3, d(e^t\alpha))
\]
where $\alpha=dz-y \, dx$ and $t$ is the variable on the first $\R$ factor. Inside the symplectization the manifold $L=\R\times \Lambda$ is a Lagrangian cylinder. 

As in Section~\ref{firstdef}, let $\{a_1,\ldots, a_n\}$ be the (generically finite) set of Reeb chords of $\Lambda$, and define $\AlgL = \Z\langle a_1,\ldots, a_n,t^{\pm 1}\rangle$.
The grading on $\AlgL$ is defined by choosing paths $\gamma_i$ as before, but now they are paths in $\Lambda$ that start at the positive end of a Reeb chord, end at the negative end of the Reeb chord, and do not pass through $*$. (Notice that these $\gamma_i$ project to the paths used in Section~\ref{firstdef}.) One can now define the gradings on the generators using the Conley--Zehnder index  associated to the $\gamma_i$ \cite{EES}, but in our current setup this is almost exactly the same as the definition given in Section~\ref{firstdef}, so we will just take the gradings from there. 

To define the boundary map for the DGA we need an almost complex structure $J$ on $(\R\times \R^3, d(e^t\alpha))$ that is compatible with $d(e^t\alpha)$. Here we can take $J:T(\R\times \R^3)\to T(\R\times \R^3)$ to be 
\begin{align*}
%J(\partial_x)&= \partial_y-x\partial_z,\\
%J(\partial_y)&= -x\partial_t -\partial_x,\\
J(\partial_x)&= \partial_y+y\partial_t,\\
J(\partial_y)&= -\partial_x-y\partial_z,\\
J(\partial_z)&=-\partial_t,\\
J(\partial_t)&=\partial_z.
\end{align*}

As before we consider $D^2_n= D^2-\{x, y_1, \ldots, y_n\}$ where $D^2$ is the unit disk in $\C$ and $x, y_1, \ldots, y_n$ are points in its boundary appearing in counterclockwise order. 
We call a map $u:D^2_n\to \R\times \R^3$ {\em $J$-holomorphic} if
\[
J\circ du= du\circ j
\]
where $j$ is the standard complex structure on $\C$. 

We will also need our maps to have nice asymptotics near the punctures. To specify this we write $u_\R$ and $u_{\R^3}$ for $u$ composed with the projections of $\R\times \R^3$ to its first and second factors respectively. Let $p$ be one of the punctures on $\partial D^2_n$ and parameterize a neighborhood of $p$ by $(0,\infty)\times [0,1]$ with coordinate $(s,t)$. Let $a(t)$ be the parameterized Reeb chord $a$; then we say $u$ is is {\em asymptotic to $a$ at $\pm \infty$} if 
\begin{align*}
\lim_{s\to \infty} u_\R (s,t)&=\pm \infty\\
\lim_{s\to \infty} u_{\R^3}(s,t)&= a(t).
\end{align*}

Now if $a, b_1, \ldots, b_n$ are points in $\{a_1,\ldots, a_n\}$ then we define the set 
\[
\M(a;b_1,\ldots, b_n) = \{u: (D^2_n,\partial D^2_n)\to (\R\times \R^3, \R\times \Lambda): \text{ satisfying (1) -- (4)}\}/\sim,
\]
where $\sim$ is holomorphic reparameterization, and 
\begin{enumerate}
\item $u$ is $J$-holomorphic,
\item $u$  has finite energy:
\[
\int_{D^2_n} u^* d\alpha<\infty,
\]
\item near $x$, $u$ is asymptotic to $a$ at $\infty$,
\item near $y_i$, $u$ is asymptotic to $b_i$ at $-\infty$. 
\end{enumerate}
For a generic choice of $\Lambda$ one can show that $\M(a;b_1,\ldots, b_n)$ is a manifold of dimension 
$|a|-\sum_{i=1}^n |b_i|$.
We also notice that $\M(a;b_1,\ldots, b_n)$ has a symmetry: given $u\in \M(a;b_1,\ldots, b_n)$, adding any constant to $u_\R$ gives another element in $\M(a;b_1,\ldots, b_n)$, and thus we have an $\R$ action on $\M(a;b_1,\ldots, b_n)$. 

Given $u\in \M(a;b_1,\ldots, b_n)$, the image of $\partial D_l^2$ is a union of $n+1$ paths $\eta_0, \ldots, \eta_n$ in $\R\times \Lambda$ where $\eta_0$ is the path parameterized by the interval in $\partial D_n^2$ starting at $x$ and $\eta_i$ is the one starting at $y_i$. We define $t(\eta_i)$ to be $t^k$ where $k$ is the number of times $\eta_i$ crosses $\R\times *$ counted with sign. The word associated to $u$ is 
\[
w(u)= t(\eta_0)b_1 t(\eta_1)b_2\cdots b_n t(\eta_n).
\]
 There is also a sign $\epsilon(u)$ that can be associated to $u$ using {\em coherent orientations}, see \cite{ENS,EES-ori}. This sign is somewhat complicated to describe and will not be essential to us here so we refer to \cite{EES-ori} for details.  
We finally define the differential of $a\in \{a_1,\ldots, a_n\}$ to be 
\[
\partial_\Lambda a = \sum\epsilon(u) w(u),
\]
where the sum is taken over all $u \in \M(a;b_1,\ldots, b_n)/\R$ where $n\geq 0$ and  $b_1, \ldots, b_n\in \{a_1,\ldots, a_n\}$ such that $|a|-\sum_{i=1}^n b_i=1$. 
As before, we define $\partial_\Lambda t=\partial_\Lambda t^{-1}=0$ and extend $\partial_\Lambda$ to all of $\AlgL$ by the signed Leibniz rule.

If we let $\pi_{xy}:\R\times \R^3\to \R^2$ be the projection map, then one may easily check that any element $u\in \M(a;b_1,\ldots, b_n)$ will project to an element $\pi_{xy}\circ u$ in $\Delta(a;b_1,\ldots, b_n)$ \cite{ENS}. From this observation and the above discussion it should be clear that the DGA just defined is equivalent to the DGA from Section~\ref{firstdef}.
\begin{remark}
We note that the original definition of $(\AlgL,\partial_\Lambda)$ was purely combinatorial, while the above described definition requires some difficult analysis to rigorously define $\M(a;b_1,\ldots, b_n)$. Despite the increased difficulty in the new definition, this is what must be used to see how the DGAs of Lagrangian cobordant Legendrian knots are related. In addition, the analysis needed for the latter definition is precisely what is needed to generalize the Chekanov--Eliashberg DGA to higher dimensions. 
\end{remark}

\subsection{Extensions of the Chekanov--Eliashberg DGA}\label{extensions} 

According to the general picture of Symplectic Field Theory \cite{EGH}, Legendrian contact homology should be defined for Legendrian submanifolds in any contact manifold $Y$. In general the algebra will be generated not just by Reeb chords but also by closed Reeb orbits in $Y$, and this gives LCH a module-like structure over the closed contact homology of $Y$. (One can remove the need to consider closed contact homology if $Y$ has no closed Reeb orbits, as is the case in $\R^3$ or more generally $X\times\R$, or by using an exact symplectic filling of $Y$ to map the closed contact homology of $Y$ to the base field.) However, the analytical underpinnings necessary to show that LCH is indeed well-defined in general are a work in progress.
Here we briefly discuss a few settings besides $\R^3$ where the Chekanov--Eliashberg DGA and LCH has been rigorously defined, both in dimension $3$ and in higher dimensions.

In dimension three, the first example of such a generalization was in Sabloff's thesis, \cite{Sab}. Here LCH was defined for circle bundles over surfaces with contact structures that are transverse to the fibers of the bundle and invariant under the natural $S^1$ action. The definition in this case looks at the projection of the Legendrian to the base manifolds and proceeds in a similar fashion to our presentation in Section~\ref{firstdef}. The main difference is that each double point in the projection corresponds to infinitely many generators of the algebra (since there are infinitely many Reeb chords that project to this double point). The differential also counts immersed polygons, but again, there are some restrictions depending on what Reeb chords one is considering for a given double point. 
Generalizing Sabloff's work, Licata and Sabloff \cite{Lic,LicSab} defined LCH for Legendrian knots in the universally tight contact structures on lens spaces $L(p,q)$ and Seifert fibered spaces with suitable contact structures. The definition in these cases are similar to those given in \cite{Sab} except that care must be taken with the topology coming form the singular fibers. 

In another direction, Ekholm and the second author \cite{EN} gave a combinatorial definition of LCH in connected sums of $S^1\times S^2$, building on a construction of Traynor and the second author \cite{NT} for LCH in the $1$-jet space $J^1(S^1)$ (this latter space, which is topologically $S^1\times\R^2$, is a local model for a neighborhood of any Legendrian knot, and is also contactomorphic to the unit cotangent bundle of $\R^2$). The contact $3$-manifolds $\#^k(S^1\times S^2)$ considered in \cite{EN} naturally appear as the boundary of Weinstein 4-manifolds, and LCH in this setting is useful when applying surgery formulas from \cite{BEE} (see Section~\ref{LCHandWeinstein} below). The algebra developed in \cite{EN} also appears in the work of An and Bae \cite{AB} defining the DGA for Legendrian graphs in $\R^3$.

In higher dimensions, Ekholm, Sullivan, and the first author gave a rigorous definition of LCH for Legendrian submanifolds in the standard contact structure on $\R^{2n+1}$ in \cite{EES}, and showed that it could be used to distinguish many Legendrian submanifolds that were ``formally isotopic'' in \cite{EES2}. In \cite{EES-LCH} the same authors extended this definition to Legendrian submanifolds of $X\times \R$ where $X$ is an exact symplectic manifold with symplectic structure $d\lambda$ and the contact structure is $\ker(dz+\lambda)$ where $z$ is the coordinate on $\R$. Once again, in all these situations the LCH is defined by projecting to $X$ and generating an algebra by the double points of the projection. The differential is defined by counting holomorphic curves, instead of immersed polygons as above. 

% ------------------------------------------------------------------------------------------------------
\section{Augmentations and Linearized LCH}
\label{sec:linear}

To readers that are more familiar with Morse or Floer homologies than with Legendrian contact homology, LCH has a major drawback that turns out to have its own advantages. Unlike many Floer complexes, the Chekanov--Eliashberg DGA is not finite rank,
even in fixed degree: a single Reeb chord generator $a$ in degree $0$ yields infinitely many generators $a^n$, all of degree $0$, for the DGA as a $\Z$-module. This can readily persist in homology: the graded pieces of LCH are often infinite dimensional, and so the graded rank of LCH would have limited utility even if this were easy to compute (which it is not in general). 

A solution to this problem, due to Chekanov, is to use an augmentation of the DGA to produce a finite-dimensional linear complex, whose homology, linearized LCH, is invariant in a suitable sense. The multiplicative structure on the DGA, which descends to homology, then produces additional interesting algebraic structures on linearized LCH, in the form of $A_\infty$ operations. In this section we describe this story, as well as some interesting connections to another collection of Legendrian invariants known as rulings.

\subsection{Augmentations and linearizations}\label{subsec:linear}
An {\em augmentation} of the Chekanov--Eliashberg DGA $(\AlgL, \partial_\Lambda)$ to a unital ring $S$ is a DGA chain map 
\[
\epsilon:\thinspace (\AlgL, \partial_\Lambda) \to (S, 0),
\]
where $S$ lies entirely in degree $0$ and has the trivial differential. Notice that this implies that $\epsilon(1)=1$, $\epsilon\circ \partial_\Lambda=0$, and $\epsilon$ sends elements of nonzero degree to 0. In addition, since $\epsilon(t)$ must be sent to an invertible (and thus nonzero) element of $S$, this also implies that $\rot(\Lambda) = 0$.

\begin{remark}
More generally, for $\Lambda$ having arbitrary rotation number, and any integer $\rho$ dividing $2\rot(\Lambda)$, one can define a \textit{$\rho$-graded augmentation} to be a DGA map $(\A_\Lambda,\d_\Lambda) \to (S,0)$, where $S$ is in degree $0$ as before but now $\A_\Lambda$ is given the grading over $\Z/\rho\Z$ induced by its grading over $\Z$. That is, $\epsilon$ now only needs to send elements of degree not divisible by $\rho$ to $0$. The cases of most interest are when $\rho=1$ ($\epsilon$ is ``ungraded''), $\rho=2$ ($\epsilon$ is $2$-graded), and $\rho=0$ (this recovers the original notion of augmentation). Unless otherwise specified, all augmentations will be $0$-graded to simplify the exposition, although a version of much of the discussion below still holds for general $\rho$-graded augmentations.
\end{remark}

Not all $(\AlgL, \partial_\Lambda)$ admit augmentations, but admitting them is a property of the stable tame isomorphism class of the DGA. We will now see how to use augmentations to ``linearize'' $(\AlgL, \partial_\Lambda)$ and illuminate other structures. To this end, let $\Lambda$ be a Legendrian knot with Reeb chords $a_1,\ldots, a_n$, and let $\kk$ be a field (a commutative unital ring would also work). If $\epsilon:\thinspace (\AlgL, \partial_\Lambda) \to (\kk, 0)$ is an augmentation, then set 
\[
\AlgL^\epsilon= \frac{\AlgL\otimes \kk}{(t=\epsilon(t))}.
\]
As an algebra, $\AlgL^\epsilon$ is simply the tensor algebra over $\kk$ generated by $a_1,\ldots,a_n$. This now inherits a differential $\d$ from the differential $\d_\Lambda$ on $\AlgL$: replace each occurrence of $t$ in $\d_\Lambda a_i$ by $\epsilon(t) \in \kk^\times$ to get the new differential $\d a_i$.
Now let $A$ be the graded $\kk$-vector space spanned by Reeb chords $a_1,\ldots, a_n$, so that
\[
\AlgL^\epsilon= \bigoplus_{n\geq 0} A^{\otimes n}.
\]
The augmentation $\epsilon$ defines an automorphism
$
\phi^\epsilon: \AlgL^\epsilon \to \AlgL^\epsilon
$ 
sending each generator $a\in \{a_1,\ldots, a_n\}$ to $\phi^\epsilon(a)=a+\epsilon(a)$, and conjugating by $\phi^\epsilon$ we get a new differential
$
\partial^\epsilon= \phi^\epsilon\circ \partial\circ (\phi^\epsilon)^{-1}
$
on $\AlgL^\epsilon$. It is easy to check that the constant term of $\partial^\epsilon (a)$ for each Reeb chord $a$ is precisely $(\epsilon\circ \partial)(a)=0$: that is, $(\AlgL^\epsilon,\d^\epsilon)$ is \textit{augmented}. If we define
$(\AlgL^\epsilon)^k= \bigoplus_{n\geq k} A^{\otimes n} \subset \AlgL^\epsilon$, then $\partial^\epsilon$ maps $(\AlgL^\epsilon)^k$ to itself for all $k \geq 0$. In particular, $\d^\epsilon$ induces a map
\[
\partial^\epsilon_1 :\thinspace \left(\frac{(\AlgL^\epsilon)^1}{(\AlgL^\epsilon)^2}\right) \to \left(\frac{(\AlgL^\epsilon)^1}{(\AlgL^\epsilon)^2}\right).
\]
Since $(\AlgL^\epsilon)^1/(\AlgL^\epsilon)^2 \cong A$, we find that $\partial^\epsilon_1$ maps $A$ to itself and satisfies $(\partial^\epsilon_1)^2=0$.
Thus $(A,\partial^\epsilon_1)$ is a differential vector space over $\kk$; its graded homology is called the {\em linearized (Legendrian) contact homology of $\Lambda$ with respect to $\epsilon$} and is denoted $LCH^\epsilon_*(\Lambda)$. 

It turns out that the linearized homology itself is not an invariant of $\Lambda$, especially as it may depend on the particular augmentation (see \cite{MS}), but it is easy to fix this problem.
\begin{theorem}[\cite{Che}]
The collection
\[
\{LCH^{\epsilon}_*(\Lambda): \text{$\e$ is an augmentation } (\AlgL,\partial_\Lambda) \to (\kk,0) \}
\]
is an invariant of $\Lambda$ up to Legendrian isotopy. Put another way,
the set of Poincar\'e polynomials 
\[
P^\epsilon(z)=\sum_{i=-\infty}^\infty \dim_{\kk}\, (LCH^{\epsilon}_i(\Lambda)) z^i
\]
over all augmentations $\epsilon :\thinspace (\AlgL,\partial_\Lambda) \to (\kk,0)$ is an invariant of $\Lambda$.
\end{theorem}
\begin{example}
One may easily compute all augmentations to $\Z_2$ for the DGAs of the Chekanov examples computed in Example~\ref{ex:chex}. The Poincar\'e polynomials for $\Lambda_1$ are all of the form
\[
z^{-2} +z + z^2,
\]
while for $\Lambda_2$ there is a unique augmentation, which has Poincar\'e polynomial
\[
2+z.
\]
Thus we see the linearized contact homology distinguishes those two examples. 
\end{example}

\begin{example}
More generally, here we present the Poincar\'e polynomials for the Legendrian twist knots in Figure~\ref{fig:chekanovfamilies}, cf.\ \cite{EFM}. Let $\Lambda$ be a knot as shown in Figure~\ref{fig:chekanovfamilies}, and let $k$ and $l$ denote the number of $Z$'s and $S$'s in the box in that diagram, where $k+l=m$. Then for any augmentation of $\Lambda$, the Poincar\'e polynomial is
$z+z^{k-l}+z^{l-k}$. It follows that linearized contact homology detects $|k-l|$ and thus (for fixed $m$) the unordered pair $\{k,l\}$. 
\end{example}

\begin{remark}
If we subtract $z$ from each of the above Poincar\'e polynomials, we obtain polynomials that are symmetric under interchanging $z \leftrightarrow z^{-1}$. This phenomenon is true in general and is known as Sabloff duality \cite{Sabloff-duality}. In its simplest form, Sabloff duality says that $\dim \LCH^\epsilon_k = \dim \LCH^\epsilon_{-k}$ except when $k = \pm 1$, and $\dim \LCH^\epsilon_1 = \dim \LCH^\epsilon_{-1}+1$. This can be upgraded to an exact triangle relating $LCH^\epsilon_*$, its dual $LCH_\epsilon^*$ (see below), and the homology of $\Lambda$; see \cite{EESduality}. Sabloff duality has been reinterpreted in \cite{NRSSZ} as a Poincar\'e-type duality between positive and negative augmentation categories (see Section~\ref{sec:augcat} below), and indeed in the case where $\epsilon$ comes from a filling $L$ (see Section~\ref{sec:fillings} below) it is precisely Poincar\'e duality for $L$.
\end{remark}

It will be useful to dualize this discussion and talk about linearized cohomology. To this end we can set $A^\vee=\Hom(A, \kk)$ and let $\delta^\epsilon_1$ be the dual of the map $\partial^\epsilon_1: A\to A$. If $A$ is generated by $a_1,\ldots,a_n$, then we denote the dual basis fo $A^\vee$ by $a_1^\vee,\ldots,a_n^\vee$ and grade them by $|a^\vee_i|=|a_i|+1$. (This grading shift is for compatibility with $A_\infty$ conventions, cf.\ Definition~\ref{def:Ainfty} below.) As usual we have $\delta^\epsilon_1\circ \delta^\epsilon_1=0$, and so we can consider the cohomology of $(A^\vee,\delta_1^\epsilon)$. This is called the {\em linearized contact cohomology with respect to $\epsilon$} and denoted $LCH^*_\epsilon(\Lambda)$.
The Universal Coefficient Theorem implies that the linearized cohomology over a field contains the same information as the linearized homology, 
but we will see that the linearized cohomology can be naturally endowed with significantly more structure. 

\begin{example}
In Example~\ref{ex:trefoil} we computed the DGA for the Legendrian trefoil in Figure~\ref{fig:trefoil}. 
\label{ex:trefoil-lin}
One can compute that there are five augmentations of the DGA to $\Z_2$. Let $\epsilon$ be the augmentation that sends $a_3$ to 1 and every other generator to 0. Then the induced differential $\partial^\epsilon$ is given by 
\begin{align*}
\partial^\epsilon a_1 &= a_3+a_5+ a_5a_4 + a_5a_4a_3,\\
\partial^\epsilon a_2 &= a_3+a_5+a_4a_5+ a_3a_4a_5,\\
\partial^\epsilon a_3 &= \partial^\epsilon a_4=\partial^\epsilon a_5=0.
\end{align*}
The only nontrivial linear terms are $\partial^\epsilon_1 a_1=a_3+ a_5$ and $\partial^\epsilon_1 a_2=a_3+a_5$. Thus the dual map is 
\begin{align*}
\delta_1^\epsilon a_3^\vee &= a_1^\vee + a_2^\vee\\
\delta_1^\epsilon a_5^\vee &= a_1^\vee + a_2^\vee\\
\delta_1^\epsilon a_1^\vee &= \delta_1^\epsilon a_2^\vee=\delta_1^\epsilon a_4^\vee=0
\end{align*}
and we have $LCH^2_\epsilon(\Lambda) \cong \Z_2$, $LCH^1_\epsilon(\Lambda) \cong (\Z_2)^2$. (This is in fact true for all five augmentations to $\Z_2$.)
\end{example}

\subsection{Augmentations and $A_\infty$ algebras}\label{ssec:a-infty}

Recall that $A$ is the $\kk$-vector space generated by Reeb chords $a_1,\ldots,a_n$. Since $\d^\epsilon$ has no constant terms, $\d^\epsilon$ maps $A$ to $\oplus_{n \geq 1} A^{\otimes n}$, and we can write
\[
\d^\epsilon = \d_1^\epsilon + \d_2^\epsilon + \cdots
\]
where $\d_n^\epsilon : A \to A^{\otimes n}$ is the map consisting of degreee $n$ terms in $\d^\epsilon$. The differential for linearized contact cohomology is the dual of $\d_1^\epsilon$; dualizing $\d_n^\epsilon$ for $n \geq 1$ gives $A$ the structure of an $A_\infty$ algebra.

\begin{definition}
An \textit{$A_\infty$ algebra} is a graded $\kk$-vector space $V$ together with a sequence of operations $m_n :\thinspace V^{\otimes n} \to V$, $n \geq 1$, of degree $1-n$, satisfying the $A_\infty$ relations:
\label{def:Ainfty}
\begin{align*}
m_1(m_1(v_1)) &= 0 \\
m_1(m_2(v_1,v_2)) &= m_2(m_1(v_1),v_2)+(-1)^{|v_1|}m_2(v_1,m_1(v_2)) \\
m_1(m_3(v_1,v_2,v_3)) &= m_2(m_2(v_1,v_2),v_3) -m_2(v_1,m_2(v_2,v_3)) \\
&\quad -m_3(m_1(v_1),v_2,v_3)-(-1)^{|v_1|}m_3(v_1,m_1(v_2),v_3))\\
&\quad -(-1)^{|v_1|+|v_2|}m_3(v_1,v_2,m_1(v_3))
\end{align*}
and generally
\[
\sum_{r+s+t=n} \pm m_{r+1+t}(1^{\otimes r} \otimes m_s \otimes 1^{\otimes t}) = 0
\]
for $n \geq 1$. (See e.g.\ \cite{NRSSZ} for an explicit choice of signs that is adapted for the setting of LCH.)
\end{definition}

\begin{proposition}
$(A^\vee,m_n = (\d^\e_n)^\vee)$ forms an $A_\infty$ algebra. Here $(\d^\e_n)^\vee :\thinspace (V^\vee)^{\otimes n} \to V^\vee$ is the dual of $\d^\e_n$.
\end{proposition}

\noindent
To see this, dualize the components of the equation $(\d^\e)^2=0$, where $(\d^\e)^2$ is viewed as a map from $V$ to $\oplus_{n=1}^\infty V^{\otimes n}$. As we have already discussed, the component of $(\d^\e)^2$ from $V$ to $V$ is $(\d^\e_1)^2$, and dualizing this gives $m_1^2 = 0$. The next component of $(\d^\e)^2$, from $V$ to $V^2$, is (up to sign)
\[
\d^\e_2 \circ \d^\e_1+(\d^\e_1 \otimes 1 + 1 \otimes \d^\e_1) \circ \d^\e_2,
\]
and dualizing this gives the second $A_\infty$ relation; and so on. Note regarding signs that there are Koszul signs implicit in the definition of $(\d^\e_n)^\vee$; see e.g.\ \cite{NRSSZ} for the explicit signs.

\begin{remark}
Here we give a more concrete description of the $A_\infty$ operations $m_n$, disregarding signs forsimplicity.
\label{rmk:a-infty}
Let $a_{i_1},\ldots,a_{i_n}$ be Reeb chord generators of $\AlgL$, and suppose that $a$ is another Reeb chord such that $\d_\Lambda a$ contains a monomial term in which $a_{i_1},\ldots,a_{i_n}$ appear in order, possibly interspersed with other $a$ generators or powers of $t$. In this monomial, replace every appearance of $t^{\pm 1}$ by $\e(t)^{\pm 1} \in \kk$, resulting in a coefficient $\alpha\in\kk$ times a product of Reeb chords:
\[
\alpha \,\mathbf{a}_0 a_{i_1} \mathbf{a}_1 a_{i_2} \mathbf{a}_2 \cdots \mathbf{a}_{n-1} a_{i_n} \mathbf{a}_n,
\]
where each $\mathbf{a}_j$ represents a (possibly empty) word of Reeb chords. Then in the twisted differential $\d^\e(a)$, there is a term where each of $\mathbf{a}_1,\ldots,\mathbf{a}_n$ is replaced by its value under $\e$, resulting in a contribution to $\d^\e_n(a)$ of $\alpha \e(\mathbf{a}_0) \cdots \e(\mathbf{a}_n) a_{i_1} a_{i_2} \cdots a_{i_n}$. Dualizing gives
\[
m_n(a_{i_n}^\vee,\ldots,a_{i_1}^\vee) = \alpha \e(\mathbf{a}_0) \cdots \e(\mathbf{a}_n) a^\vee + \cdots.
\]
Here the convention on the order of inputs is the reverse of the order in $\d^\e_n(a)$; this allows for compatibility with standard $A_\infty$-category conventions (see Section~\ref{ssec:augcat} below).

\begin{figure}[htb]{\small
\begin{overpic}{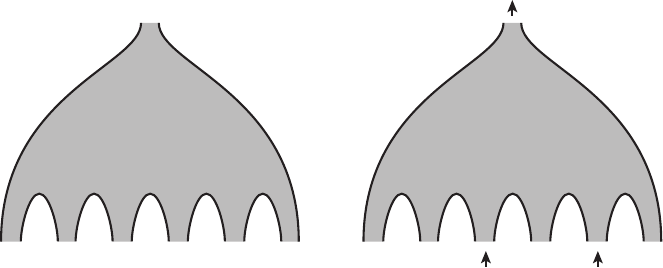}
\put(69, 123){{\color{red} $a$}}
\put(54, 2){{\color{red} $a_{i_1}$}}
\put(109, 2){{\color{red} $a_{i_2}$}}
\put(251, 123){{\color{red} $a^\vee$}}
\put(238, 2){{\color{red} $a_{i_1}^\vee$}}
\put(292, 2){{\color{red} $a_{i_2}^\vee$}}
\put(1, 2){{$a_{j_1}$}}
\put(29, 2){{$a_{j_2}$}}
\put(82, 2){{$a_{j_3}$}}
\put(136, 2){{$a_{j_4}$}}
\put(171, 2){{$\epsilon(a_{j_1})$}}
\put(198, 2){{$\epsilon(a_{j_2})$}}
\put(252, 2){{$\epsilon(a_{j_3})$}}
\put(305, 2){{$\epsilon(a_{j_4})$}}
\end{overpic}}
\caption{
A disk with positive end at $a$ and negative ends including $a_{i_1}$ and $a_{i_2}$ contributes an $a^\vee$ term to $m_2(a_{i_2}^\vee,a_{i_1}^\vee)$.}
\label{fig:a-infty}
\end{figure}
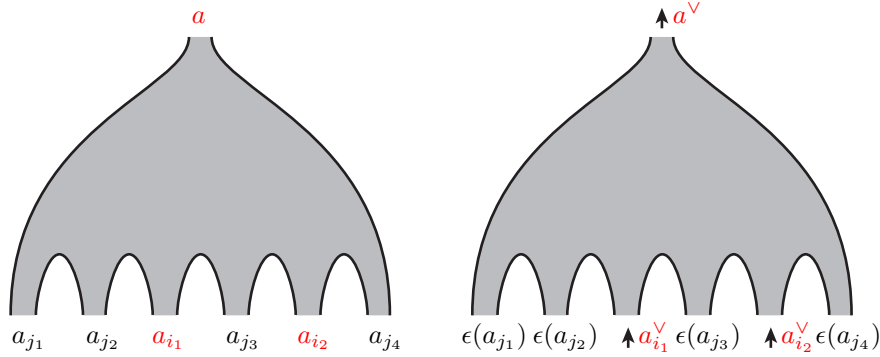

For an illustration, see Figure~\ref{fig:a-infty}. Here the term $\d_\Lambda a = a_{j_1}a_{j_2}a_{i_1}a_{j_3}a_{i_2}a_{j_4} + \cdots$ dualizes to $m_2(a_{i_2}^\vee,a_{i_1}^\vee) = \e(a_{j_1})\e(a_{j_2})\e(a_{j_3})\e(a_{j_4})a_k^\vee + \cdots$. (In fact, this single disk could make $15$ contributions to $m_2$, corresponding to the $6\choose 2$ ways to choose two inputs from  $a_{j_1},a_{j_2},a_{i_1},a_{j_3},a_{i_2},a_{j_4}$.)
\end{remark}

Given an $A_\infty$ algebra $(V,m_n)$, we can define the graded homology $H(V,m_1)=\ker m_1/\im m_1$, since $m_1^2=0$ by the first $A_\infty$ relation. By the second $A_\infty$ relation, we can view $m_2$ as a multiplication operation on $V$ for which the differential $m_1$ satisfies the Leibniz rule. It follows that $m_2$ descends to a well-defined product on $H(V,m_1)$. Furthermore, although $m_2$ is not necessarily associative as a product on $V$, the third $A_\infty$ relation implies that it is associative on $H(V,m_1)$. 

We conclude that $H(V,m_1)$ is a ring with multiplication given by $m_2$. In the case of interest to us, $\LCH^*_\e$ is a ring where the product structure comes from the second order terms in the differential $\d^\e$. This picture is entirely analogous to how the cup product induces multiplication on the singular cohomology for topological spaces.

\begin{figure}[htb]{\small
\begin{overpic}{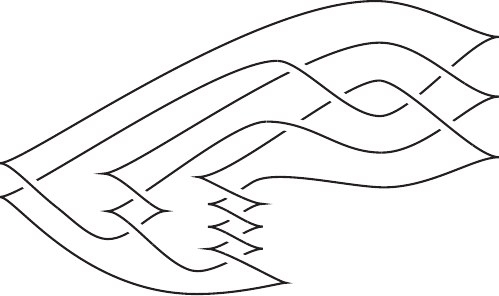}
\end{overpic}}
\caption{
A knot that can be distinguished from its Legendrian mirror by the product on linearized cohomology.
}
\label{fig:Legmirror}
\end{figure}

Some Legendrian knots cannot be distinguished by their linearized cohomologies $\LCH^*_\e$ but can be distinguished by the product on $\LCH^*_\e$. An example of such a pair of knots is the Legendrian knot $\Lambda$ shown in Figure~\ref{fig:Legmirror}, along with its ``Legendrian mirror'' obtained by reflecting the front projection for $\Lambda$ in the $x$ axis. These two knots have isomorphic $\LCH^*_\e$ but their products are opposite: $m_2(v_1,v_2)$ in one is $m_2(v_2,v_1)$ in the other. See \cite{CKESW} for this computation, along with more general families of Legendrian knots that require the use of higher order products $m_n$ to tell them apart.

\subsection{Rulings and augmentations}\label{rulingaugs}
In this section we explore a geometric invariant of front diagrams that is closely connected to augmentations. Given the front projection of a Legendrian knot $\Lambda$ we call the {\em arcs} of $F(\Lambda)$ the closures of the components of $F(\Lambda)$ minus the cusps and crossings. A {\em $\rho$-graded ruling} of $F(\Lambda)$ is a partition $R=\{R_1,\ldots, R_k\}$ of the arcs of $F(\Lambda)$ so that 
\begin{enumerate}
\item each $R_i$ bounds a disk,
\item  each $R_i$ contains one left and one right cusp, 
\item the crossings where $R_i$ is not smooth are called switches and the grading of a switch must be divisible by $\rho$, and
\item The disks associated to two $R_i$'s at a switch must locally have interiors nested or disjoint. See Figure~\ref{fig:switches}.
\end{enumerate}
\begin{figure}[htb]{\small
\begin{overpic}{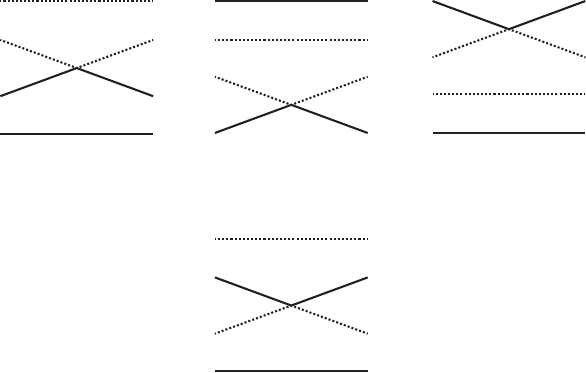}
\end{overpic}}
\caption{Top row shows allowable switches. The bottom row shows a disallowed switch. In each case the solid arcs are paired by the ruling, as are the dotted arcs.}
\label{fig:switches}
\end{figure}

To a $\rho$-graded ruling $R$ of $F(\Lambda)$ we associate the number 
\[
\theta(R)=k-s,
\]
where $k$ is the number of components of $R$ (which is equal to half the number of cusps of $F(\Lambda)$) and $s$ is the number of switches in $R$. We can now define the {\em complete $\rho$-graded ruling invariant} to be the multiset
\[
\Theta_\rho(\Lambda)=\{\theta(R)\,:\, R \text{ a $\rho$-graded ruling of $F(\Lambda)$}\}.
\]
One may check that the following is true. 
\begin{theorem}[Chekanov and Pushkar \cite{ChP}]
For any $\rho$ that divides $2 \rot(\Lambda)$, the complete $\rho$-graded ruling invariant $\Theta_\rho(\Lambda)$ is an invariant of the Legendrian isotopy class of $\Lambda$. 
\end{theorem}
\noindent
For example, one can use this invariant (with $\rho=0$) to distinguish the Chekanov examples from Figure~\ref{fig:CHex}.

Rulings turn out to be closely connected to augmentations. Indeed, by combined work of Fuchs, Ishkhanov, and Sabloff, we have the following result.

\begin{theorem}[\cite{Fuchs,FI,Sabloff}]
For any $\rho$ dividing $2\rot(\Lambda)$, 
\label{thm:ruling-aug}
the Legendrian knot $\Lambda$ has a $\rho$-graded augmentation to $\Z_2$ if and only if it has a $\rho$-graded ruling.
\end{theorem}

Theorem~\ref{thm:ruling-aug} has been extended by Leverson \cite{Leverson}, who proved that the existence of an augmentation to any field is equivalent to the existence of a ruling. This is not true if we replace ``field'' by an arbitrary unital ring; see Section~\ref{ssec:rep} below.

It turns out there is a precise correspondence between rulings and augmentations. To state the correspondence we need a bit more notation. We restrict our discussion to augmentations to $\Z_2$, though see \cite{HenryRutherford} for a generalization to arbitrary finite fields.

Because of DGA stabilizations, the number of $\rho$-graded augmentations of $(\AlgL, \partial_\Lambda)$ to $\Z_2$ is not an invariant of $\Lambda$ up to Legendrian isotopy, but there is a normalized count that is. More specifically, given any $\rho$ that divides $2\rot(\Lambda)$, let $a_k$ be the number of generators of $\AlgL$ (in the front projection, that is, crossings and right cusps)
with grading $k$ modulo $\rho$. The shifted Euler characteristic of $(\AlgL, \partial_\Lambda)$ when $\rho=0$ is defined to be
\[
\chi_0^*(\AlgL)= \sum_{k\geq 0} (-1)^ka_k + \sum_{k<0} (-1)^{k+1} a_k
\]
and if $\rho$ is odd then it is
\[
\chi_\rho^*(\AlgL)= \sum_{k= 0}^{\rho-1} (-1)^ka_k.
\]
We now define the {\em normalized $\rho$-graded augmentation number}\footnote{There is also a related but more categorical notion of normalized augmentation number in terms of the cardinality of the augmentation category; see \cite{NRSS}.} to be
\[
Aug_\rho(\Lambda)= 2^{-\chi_\rho^*(\AlgL)/2} \cdot \text{(number of $\rho$-graded augmentations of $\AlgL$)}.
\] 
This number can easily be checked to be an invariant of $\Lambda$ up to Legendrian isotopy, and for instance it provides yet another way to distinguish the Chekanov knots (Example~\ref{ex:chex}): 
$Aug_0(\Lambda_1) = \sqrt{2}$ while $Aug_0(\Lambda_2) = 3/\sqrt{2}$. We can now state the explicit connection between rulings and augmentations. 
\begin{theorem}[\cite{NgSab}]
Given a Legendrian knot $\Lambda$ and a number $\rho$ that divides $2\rot(\Lambda)$ and is either 0 or odd, then there is a many-to-one correspondence between $\rho$-graded augmentations of $(\AlgL, \partial_\Lambda)$ and $\rho$-graded rulings of $F(\Lambda)$. More specifically, there are $2^{\theta(R)+\chi_\rho^*(\AlgL)/2}$ $\rho$-graded augmentations corresponding to each $\rho$-graded ruling $R$. 
\end{theorem}
This theorem moreover allows one to determine the normalized count of augmentations from the rulings as follows: if $\rho$ divides $2\rot(\Lambda)$ and is 0 or odd then 
\[
Aug_\rho(\Lambda)=\sum_{\theta\in \Theta_\rho(\Lambda)} 2^{\theta/2}.
\]

In 2005, Rutherford \cite{Ruth} discovered a beautiful connection between (ungraded) rulings and topology that says, among other things, that $\Theta_1(\Lambda)$ only depends on the underlying topological knot type of $\Lambda$ and $\tb(\Lambda)$. To state his result we first recall the Kauffman and HOMFLY polynomials of a knot $K$. The Kauffman polynomial $F_K(a,z)$ of a knot $K$ is 
defined as 
\[
F_K(a,z) = a^{-w(D_K)}A_{D_K}(a,z),
\]
where $w(D_K)$ is the writhe of the knot a diagram $D_K$ for $K$ 
and $A_{D_K}$ is a polynomial defined for the diagram $D_K$, uniquely characterized by the skein relations 
\[
D_{K_+}-D_{K_-}= z(D_{K_0}-D_{K_\infty}),
\]
\[
D_{S_+}=aD_A,\quad D_{S_-}=a^{-1}D_S,
\]
and $D$ of the unknot is $1$, where the diagrams are shown in Figure~\ref{fig:Kauffman}.
\begin{figure}[htb]{\small
\begin{overpic}{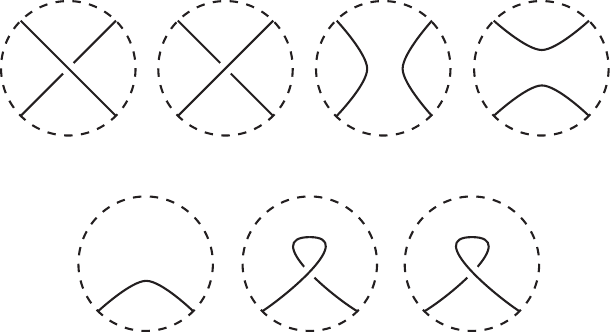}
\put(5, 85){$K_+$}
\put(80, 85){$K_-$}
\put(160, 85){$K_\infty$}
\put(235, 85){$K_0$}
\put(40, -5){$S$}
\put(117, -5){$S_-$}
\put(200, -5){$S_+$}
\end{overpic}}
\caption{Diagram regions for the skein relations for $F_K(a,z)$.}
\label{fig:Kauffman}
\end{figure}
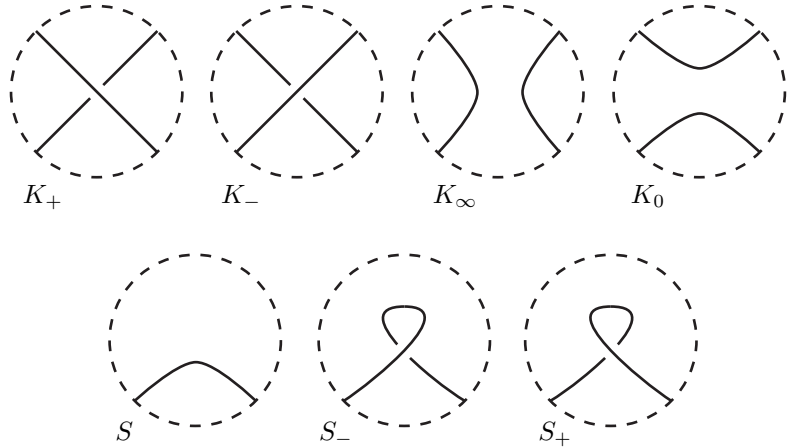
The HOMFLY polynomial $P_K(a,z)$ of a knot $K$ is similarly defined using $D_K$:
\[
P_K(a,z) = a^{-w(D_K)}B_{D_K}(a,z),
\]
where $B_{D_K}$ is a polynomial defined for the diagram $D_K$ and uniquely characterized by the skein relations 
\[
B_{K_+}-B_{K_-}= zH_{K_0},
\]
\[
B_{S_+}=a B_S,\quad B_{S_-}=a B_S,
\]
and $B$ of the unknot is 1, where the $K_+$, $K_-$ and $K_0$ are all oriented left to right in Figure~\ref{fig:Kauffman}.

There are upper bounds on $\tb$ from both the HOMFLY-PT \cite{FW,Morton} and Kauffman \cite{Rudolph} polynomials: if $\deg_a$ means the maximal degree of the polynomial in the variable $a$, then we have
\begin{align*}
\tb(\Lambda)+ |\rot(\Lambda)| &\leq - \deg_a P_\Lambda(a,z)-1 \\
\tb(\Lambda) &\leq -\deg_a F_{\Lambda}(a,z)-1.
\end{align*}
Following Rutherford we now define the {\em ruling polynomial of $\Lambda$} to be
\[
R_\Lambda(z)= \sum_{\theta\in \Theta_1(\Lambda)} z^{-\theta+1}.
\]
We can also define the {\em oriented ruling polynomial of $\Lambda$}. To this end consider the subset $\Theta_2(\Lambda)$ of $\Theta_1(\Lambda)$. One may check that these are precisely the $\theta$ that come from ``oriented rulings'', that is, rulings where we only allow switches at positive crossings in the diagram:
\[
OR_\Lambda(z)= \sum_{\theta\in \Theta_2(\Lambda)} z^{-\theta+1}.
\]
It is a result of Sabloff \cite{Sabloff} that $\Lambda$ can have an oriented ruling only if $\rot(\Lambda) = 0$. 

\begin{theorem}[Rutherford \cite{Ruth}]
For any Legendrian knot $\Lambda$ of topological type $K$,
\label{thm:Rutherford}
the ruling polynomial $R_\Lambda(z)$ and oriented ruling polynomial $OR_\Lambda(z)$ agree with the (polynomial) coefficient of $a^{\tb(\Lambda)-1}$ in $F_K(a,z)$ and $P_K(a,z)$, respectively.
\end{theorem}
Notice that this theorem says that ungraded rulings (both oriented and not) are entirely determined by the underlying knot type of the Legendrian knot and the classical invariants. So in particular one will not be able to use ungraded rulings to distinguish Legendrian knots with the same classical invariants! Also notice that an immediate corollary of the theorem is that the Kauffman bound on $\tb$ is sharp if and only if $\Lambda$ admits an ungraded ruling.

Moreover, if a Legendrian has an ungraded ruling then its Thurston--Bennequin invariant is maximal for Legendrian representatives of its knot type. 

\subsection{DGA representations}
\label{ssec:rep}

Much of the existing work on augmentations of Legendrian knots has focused on augmentations to a field. It is however also interesting to consider augmentations to other unital rings $S$ that are not fields. A particular case is when $S = \Mat_n(\kk)$, the algebra of $n\times n$ matrices over a field $\kk$. We call an augmentation $\rho :\thinspace (\A_\Lambda,\d_\Lambda) \to (\Mat_n(\kk),0)$ an \textit{$n$-dimensional representation} of the DGA $(\A_\Lambda,\d_\Lambda)$. Note that $1$-dimensional representations are precisely augmentations to $\kk$ and these all factor through the abelianization of $\A_\Lambda$. An advantage of considering higher-dimensional representations is that these allow us to use the noncommutativity of $\A_\Lambda$ in a more fundamental way, since these representations do not necessarily factor through the abelianization. Here it is useful to use the fully noncommutative DGA $\A_\Lambda$ rather than the form of the DGA that appears in for example \cite{ENS}, where $t$ commutes with Reeb chords, since stipulating that $\rho(t)$ and $\rho(a)$ commute for all Reeb chords $a$ significantly cuts down on the set of representations. For example, Theorem~\ref{thm:rep-sat} below, which relates representations of the DGA to augmentations of a satellite, is only true if we use the fully noncommutative DGA.

The existence of a representation of $(\A_\Lambda,\d_\Lambda)$, like the existence of an augmentation, is an obstruction to the DGA being trivial, and thus to $\Lambda$ being stabilized. There are Legendrian knots that have no augmentations but do have higher-dimensional representations. The earliest work on this was by Sivek \cite{Sivek}, who found a family of Legendrian torus knots of type $T(p,-q)$, where $q>p\geq 3$ and $p$ is odd, that have $2$-dimensional representations but no augmentations to $\Z_2$. In particular, the knot $8_{19} = T(3,-4)$ falls into this family:

\begin{theorem}[\cite{Sivek}]
The DGA of the Legendrian knot shown in Figure~\ref{fig:torusknot} 
\label{thm:2drep}
admits an ungraded $2$-dimensional representation but not an ungraded $1$-dimensional representation over $\Z_2$.
\end{theorem}

\begin{figure}[htb]{\small
\begin{overpic}{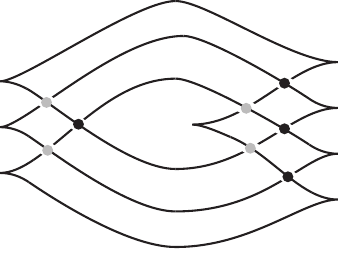}
\end{overpic}}
\caption{A Legendrian knot of type $T(3,-4)$.}
\label{fig:torusknot}
\end{figure}

\begin{proof}
For the DGA $(\A_\Lambda,\d_\Lambda)$ of this knot in the front projection (where the base point is placed anywhere), one can check that the map $\rho: \A_\Lambda \to \Mat_2(\Z_2)$ sending $t$ to $\left( \begin{smallmatrix} 1 & 0 \\ 0 & 1\end{smallmatrix}\right)$, each dark-dotted crossing to $\left( \begin{smallmatrix}0 & 0 \\ 1 & 0\end{smallmatrix} \right)$, each light-dotted crossing to $\left( \begin{smallmatrix} 0 & 1 \\ 0 & 0\end{smallmatrix} \right)$, and the cusps to $0$ satisfies $\rho \circ \d_\Lambda = 0$. On the other hand, the coefficient of $a^{tb(\Lambda)-1} = a^{-13}$ in the Kauffman polynomial $F_{T(3,-4)}(a,z)$ is $0$; so by Theorem~\ref{thm:Rutherford}, $\Lambda$ has no rulings. It follows from Theorem~\ref{thm:ruling-aug} that $\Lambda$ has no ungraded augmentations.
\end{proof}

We close this section by noting that there is a correspondence between $n$-dimensional representations of the DGA $(\A_\Lambda,\d_\Lambda)$ and augmentations of a certain Legendrian link consisting of $n$ parallel copies of $\Lambda$. More precisely, let $\Lambda^{(n)}$ denote the $n$-component Legendrian link whose front consists of $n$ copies of the front of $\Lambda$, pushed off from each other by small perturbations in the Reeb, that is, the $z$ direction, with one segment of the $n$ parallel fronts replaced by a full positive twist.
We then have the following result.

\begin{theorem}[\cite{NR}]
The DGA for $\Lambda$ has an $n$-dimensional representation over $\Z_2$
\label{thm:rep-sat}
if and only if the DGA for $\Lambda^{(n)}$ has an augmentation to $\Z_2$.
\end{theorem}

\noindent
Since the existence of an augmentation is equivalent to the existence of a ruling by Theorem~\ref{thm:ruling-aug}, one can reprove Theorem~\ref{thm:2drep} by exhibiting an ungraded ruling of $\Lambda^{(2)}$ where $\Lambda$ is the knot in Figure~\ref{fig:torusknot}; see \cite{NR} for an illustration of such a ruling.
We also remark that one can count the number of representations of the DGA of a Legendrian knot $\Lambda$ over a finite field, and this is related to rulings of satellites of $\Lambda$ (generalizing $\Lambda^{(n)}$) through colored HOMFLY-PT polynomials; see \cite{LR}.

% ------------------------------------------------------------------------------------------------------
\section{Augmentation Categories}
\label{sec:augcat}

In this section, we describe how the collection of augmentations of the DGA of a Legendrian knot can be assembled into the algebraic structure of an $A_\infty$ category, called the augmentation category. (There are in fact two categories $\Aug_-$ and $\Aug_+$, which we describe in turn.) The morphisms in the category are a generalization of the linearized contact homology and $A_\infty$-algebra operations discussed in Section~\ref{sec:linear}, and the category itself is meant to model a Fukaya category whose objects are exact fillings (see Section~\ref{sec:fillings} below). One benefit of this categorical formulation is that it yields a natural algebraic notion of equivalence for augmentations, generalizing the geometric notion of isotopy of fillings. The augmentation category also has an intriguing relation to sheaf theory; a full description lies outside the scope of this article, but we give a brief discussion at the end of this section.

\subsection{Two $A_\infty$ categories}
\label{ssec:augcat}

The $A_\infty$ algebra described in Section~\ref{ssec:a-infty} is associated to a choice of augmentation of the DGA $(\A_\Lambda,\d_\Lambda)$. One can generalize this to incorporate multiple augmentations of $(\A_\Lambda,\d_\Lambda)$, as was first observed in this context by Bourgeois and Chantraine \cite{BC}. To see this, consider a term in $\d_\Lambda a$ of the form $\alpha\mathbf{a}_0 a_{i_1} \mathbf{a}_1 a_{i_2} \mathbf{a}_2 \cdots \mathbf{a}_{n-1} a_{i_n} \mathbf{a}_n$ as in Remark~\ref{rmk:a-infty} above. If we now have not $1$ but $n+1$ augmentations $\e_0,\ldots,\e_{n+1}$, then we can replace $\mathbf{a}_0,\ldots,\mathbf{a}_n$ successively by $\e_0(\mathbf{a}_0),\ldots,\e_n(\mathbf{a}_n)$, and dualizing now gives
\begin{equation}
m_n(a_{i_n}^\vee,\ldots,a_{i_1}^\vee) = \alpha \e_0(\mathbf{a}_0) \cdots \e_n(\mathbf{a}_n) a^\vee + \cdots.
\label{eq:bc}
\end{equation}
These new $m_n$ operations depend on the choice of augmentations $\e_0,\ldots,\e_n$. Where the $m_n$ formed an $A_\infty$ algebra when all of the $\e_i$ were equal, they now form the crucial ingredients to an $A_\infty$ category.

\begin{definition}
An \textit{$A_\infty$ category} $\CC$ consists of: a set of objects $\Ob \CC$; a graded $\kk$-vector space $\Hom(\e_1,\e_2)$ for any objects $\e_1,\e_2\in\Ob\CC$; and, for $n \geq 1$ and any objects $\e_0,\ldots,\e_n\in\Ob\CC$, a map
\[
m_n :\thinspace \Hom(\e_{n-1},\e_n) \otimes \cdots \otimes \Hom(\e_1,\e_2) \otimes \Hom(\e_0,\e_1) \to \Hom(\e_0,\e_n)
\]
of degree $1-n$, such that the $A_\infty$ relations
\[
\sum_{r+s+t=n} \pm m_{r+1+t}(1^{\otimes r} \otimes m_s \otimes 1^{\otimes t}) = 0
\]
hold for $n \geq 1$.
\end{definition}

We now have the following result, whose proof (omitted here) is a formal algebraic consequence of the fact that $\d_\Lambda^2=0$.

\begin{theorem}[Bourgeois--Chantraine \cite{BC}]
Given a Legendrian knot $\Lambda \subset \R^3$ and a field $\kk$, there is an $A_\infty$ category $\Aug_-(\Lambda,\kk)$ such that:
\begin{itemize}
\item
$\Ob \Aug_-(\Lambda,\kk)$ is the set of augmentations $(\A_\Lambda,\d_\Lambda) \to (\kk,0)$;
\item
for any $\e_1,\e_2$, $\Hom(\e_1,\e_2) = \kk\langle\RR\rangle^\vee$, the dual to the $\kk$-vector space generated by the set $\RR$ of Reeb chords of $\Lambda$;
\item
the $m_n$ operations for $n\geq 1$ are given by Equation~\eqref{eq:bc}.
\end{itemize}
\end{theorem}

The $A_\infty$ category $\Aug_-(\Lambda,\kk)$ is one of two $A_\infty$ categories that can be constructed from augmentations. To set up the other category $\Aug_+$, we first reformulate the definition of $\Aug_-$, following \cite{BC}. For $n \geq 1$, the \textit{$n$-copy} $\Lambda_{(n)}$ of a Legendrian knot $\Lambda$ is the $n$-component Legendrian link given by $\Lambda$ along with $n-1$ additional copies, perturbed to be distinct from $\Lambda$ and each other by small translations in the Reeb, that is, the $z$ direction. For now we number these copies $\Lambda_1,\ldots,\Lambda_n$ from bottom to top, so that $\Lambda_k$ is the result of translating $\Lambda$ by $(k-1)\epsilon$ in the $z$ direction for $\epsilon \ll 1$. The $xy$ projection of $\Lambda_{(n)}=\Lambda_1 \cup \cdots \cup \Lambda_n$ consists of $n$ overlapping projections of $\Lambda$; to make this generic, we perturb the $xy$ projections of the components so that they intersect transversely. To do this, we choose a positively-valued Morse function $f$ on $\Lambda$, identify a tubular neighborhood of $\Lambda$ with the $1$-jet space $J^1\Lambda$, and choose $\Lambda_k$ to correspond to the $1$-jet of the function $(k-1)\epsilon f$ in $J^1\Lambda$. The result in the $xy$ projection is $n$ parallel copies of $\Lambda$ that all intersect at each critical point of $f$. We then further perturb these collections of $n\choose 2$ intersections to make them distinct from each other.  See Figure~\ref{fig:3copy} for an illustration of a $3$-copy, and \cite{BC,NRSSZ} for more details.

\begin{figure}
{\tiny
\begin{overpic}{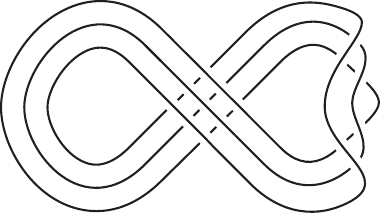}
\put(31, 21){$\Lambda_3$}
\put(23, 12){$\Lambda_2$}
\put(14, 3){$\Lambda_1$}
\end{overpic}}
\caption{
A $3$-copy of the Legendrian unknot, in the $xy$ projection. The perturbing Morse function on the unknot that yields this $3$-copy has two critical points, a maximum in the upper right of the figure eight and a minimum in the lower right. This yields two triple points in the $xy$ projection, which have further been perturbed to give two sets of $3$ crossings as shown.
}
\label{fig:3copy}
\end{figure}

Let $\RR$ and $\RR_{(n)}$ denote the set of Reeb chords of $\Lambda$ and $\Lambda_{(n)}$, respectively. Following Mishachev \cite{Mis}, we can partition $\RR_{(n)}$ into $n^2$ subsets $\RR^{ij}$, $1\leq i,j \leq n$, where $\RR^{ij}$ consists of Reeb chords that begin on $\Lambda_j$ and end on $\Lambda_i$. From the description of the $xy$ projection of $\Lambda_{(n)}$ above, we see that Reeb chords in $\RR$ fall into two types: for each crossing in $\pi_{xy}(\Lambda)$, there are $n^2$ crossings in $\pi_{xy}(\Lambda_{(n)})$, one in each $\RR^{ij}$; and for each critical point of $f$, there are $n\choose 2$ crossings in $\pi_{xy}(\Lambda_{(n)})$, one in each $\RR^{ij}$ for $i>j$. It follows that there is a one-to-one correspondence between $\RR^{ij}$ and $\RR$ for $i\leq j$, and between $\RR^{ij}$ and $\RR$ together with critical points of $f$, for $i>j$.

Now let $\d_{(n)}$ denote the LCH differential for $\Lambda_{(n)}$. For $a \in \RR^{ij}$, every term in $\d_{(n)}$ is \textit{composable}: disregarding homology coefficients, it is of the form
$a_1a_2\cdots a_k$, where $a_1\in\RR^{i i_1}$, $a_2 \in \RR^{i_1 i_2}$, $a_3 \in \RR^{i_2 i_3}$, $\ldots$, $a_k \in \RR^{i_{k-1}j}$ for some $i_1,\ldots,i_{k-1}\in\{1,\ldots,n\}$. This comes from considering the piecewise-smooth boundary of the relevant holomorphic disk, and in particular which components of the $n$-copy it lies in. The same remains true if we twist the differential by a \textit{pure} augmentation of $\Lambda_{(n)}$, defined to be an augmentation that sends all generators in $\RR^{ij}$ to $0$ for $i \neq j$.

We can now state an alternate definition for $\Aug_-(\Lambda,\kk)$. As before, the objects of $\Aug_-$ are augmentations $(\A_\Lambda,\d_\Lambda) \to (\kk,0)$, and the morphisms $\Hom(\e_1,\e_2)$ are the $\kk$-vector space generated by (duals of) Reeb chords of $\Lambda$, which we can now identify with the dual of the vector space $\kk\langle \RR^{ij} \rangle$ generated by $\RR^{ij}$ for any $i\leq j$. Let $\e_0,\ldots,\e_n$ be augmentations in $\Ob \Aug_-$. Then we can define a pure augmentation $\e = (\e_0,\ldots,\e_n)$ of $\Lambda^{(n+1)}$ by $\e(a) = \e_j(a)$ if $a \in \RR^{jj}$ and $\e(a) = 0$ if $a\in \RR^{ij}$ for $i\neq j$. The twisted differential $\d_{(n+1)}^\e$ consists of composable terms, and we can dualize it to obtain a map
\begin{equation}
m_n :\thinspace (\kk\langle \RR^{n-1,n}\rangle)^\vee \otimes \cdots \otimes (\kk\langle\RR^{12}\rangle)^\vee \otimes (\kk\langle\RR^{01}\rangle)^\vee \to 
(\kk\langle\RR^{0n}\rangle)^\vee.
\label{eq:ncopy}
\end{equation}
More precisely, a degree $n$ term $a_{j_1}\cdots a_{j_n}$ in $\d^{(n)}(a)$ for $a \in \RR^{0n}$ and $a_{j_k} \in \RR^{k-1,k}$ for $1\leq k\leq n$ dualizes to a term $a^\vee$ in $m_n(a_{j_n}^\vee,\ldots,a_{j_1}^\vee)$.

To see that this definition of $m_n$ agrees with the previous definition in Equation~\eqref{eq:bc}, the idea is to look at a holomorphic disk contributing to $\d_{(n)}$ and $m_n$, and note that in the limit that all copies of $\Lambda_{(n)}$ approach $\Lambda$, this disk approaches a disk for the original differential $\d_\Lambda$. We leave the details to the reader (or see \cite{BC,NRSSZ}).

With a minor change, this formulation of $\Aug_-$ in terms of $n$-copies allows us to define another $A_\infty$ category $\Aug_+$ that has some nicer formal properties than $\Aug_-$. The change is simply reversing the order of the components in the $n$-copy $\Lambda_{(n)}$, so that $\Lambda_1$ is on top in the $z$ direction and $\Lambda_n$ is on bottom. The $m_n$ maps are then defined as before. However, note that the sets $\RR^{ij}$ appearing in the definition of the $m_n$ maps in Equation~\eqref{eq:ncopy} are no longer in one-to-one correspondence with Reeb chords of $\Lambda$, and now have additional elements corresponding to critical points of the Morse function $f$. If we choose $f$ on the knot $\Lambda$ to have a single maximum $x$ and a single minimum $y$, then we can identify $\RR^{ij}$ with $\RR \cup \{x,y\}$. We now have the following.

\begin{theorem}[\cite{NRSSZ}]
Given a Legendrian knot $\Lambda \subset \R^3$ and a field $\kk$, there is an $A_\infty$ category $\Aug_+(\Lambda,\kk)$ such that:
\begin{itemize}
\item
$\Ob \Aug_+(\Lambda,\kk)$ is the set of augmentations $(\A_\Lambda,\d_\Lambda) \to (\kk,0)$;
\item
for any $\e_1,\e_2$, $\Hom(\e_1,\e_2) = \kk\langle\RR\cup\{x,y\}\rangle^\vee$;
\item
the $m_n$ operations for $n\geq 1$ are given as in Equation~\eqref{eq:ncopy}.
\end{itemize}
\end{theorem}

\begin{remark}
Note that whereas the definition of $\Aug_-(\Lambda,\kk)$ by Bourgeois and Chantraine (as we described at the beginning of this section) uses the DGA of $\Lambda$, the construction of $\Aug_+(\Lambda,\kk)$ we have just described involves the DGAs of not just $\Lambda$ but also its $n$-copies. However, it is also possible to deduce the $A_\infty$ category $\Aug_+(\Lambda,\kk)$ solely from the DGA of $\Lambda$ itself; see \cite[\S 4]{NRSSZ} for details. 

In higher dimensions, for Legendrian submanifolds $\Lambda$ with $\dim \Lambda >1$, one can again use \cite{BC} to construct an $A_\infty$ category $\Aug_-(\Lambda,\kk)$ from the DGA of $\Lambda$. It is expected that one can also construct the analogue of $\Aug_+(\Lambda,\kk)$ in this setting, but we caution that this likely involves more data than the DGA of $\Lambda$, in contrast to the case of $1$-dimensional Legendrian knots that we have discussed here.
\end{remark}

\subsection{Properties of $\Aug_+$}
\label{ssec:aug-properties}

There is a subtle but important distinction between $\Aug_-$ and $\Aug_+$, and it has to do with the presence of $y^\vee$ in $\Aug_+$. In the $n$-copy $\Lambda_{(n)}$, it can be checked that the only holomorphic disks with a negative corner at one of the crossings corresponding to $y$ are triangles that are ``thin'' in the sense that they lie entirely in a neighborhood of $\Lambda$. It follows from this that in $\Aug_+$, $m_2(a^\vee,y^\vee)$ and $m_2(y^\vee,a^\vee)$ are (up to sign) both equal to $a^\vee$ for any $a \in \RR\cup\{x,y\}$. More precisely:

\begin{definition}
An $A_\infty$ category $\CC$ is \textit{strictly unital} if for all $\e\in\Ob\CC$, there is a morphism $e_\e \in \Hom(\e,\e)$ such that: $m_1(e_\e) = 0$; any $m_n$ for $n\geq 3$ involving $e_\e$ is $0$; and for any $\e_1,\e_2$ and any $a\in\Hom(\e_1,\e_2)$,
\[
m_2(a,e_{\e_1}) = m_2(e_{\e_2},a) = a.
\]
\end{definition}

\begin{theorem}[\cite{NRSSZ}]
$\Aug_+(\Lambda,\kk)$ is strictly unital.
\end{theorem}

The unitality of $\Aug_+$ allows us to construct a ``usual'' category, the \textit{cohomology category} $H^*\Aug_+$, from $\Aug_+$. The objects of $H^*\Aug_+$ are the same as the objects of $\Aug_+$, and the morphisms are $\Hom_{H^*\Aug_+}(\e_1,\e_2) = H^*(\Hom_{\Aug_+}(\e_1,\e_2),m_1)$. In $H^*\Aug_+$, composition is the map induced from $m_2$ (by the $A_\infty$ relations, this is associative) and $[e_\e]$ serves as the identity morphism from $\e$ to itself. 

This leads to a notion of equivalence for augmentations: two augmentations $\e_1,\e_2$ are \textit{isomorphic} in $\Aug_+$ if there are morphisms $a \in H^*\Hom(\e_1,\e_2)$ and $a' \in H^*\Hom(\e_2,\e_1)$ such that $a' \circ a = [e_{\e_1}] \in H^*\Hom(\e_1,\e_1)$ and $a\circ a' = [e_{\e_2}] \in H^*\Hom(\e_2,\e_2)$, where $\Hom = \Hom_{\Aug_+}$. In this setting, isomorphism of augmentations actually coincides with the notion of DGA chain homotopy; see \cite[Prop.~5.17]{NRSSZ}.

\begin{remark}
If $\kk$ is a finite field, then the number of isomorphism classes of augmentations to $\kk$ is a Legendrian-isotopy invariant of $\Lambda$. This for instance gives another way to distinguish the Chekanov $m(5_2)$ knots: $\Lambda_1$ from Figure~\ref{fig:CHex} has $1$ isomorphism class of augmentations to $\Z_2$, while $\Lambda_2$ has $3$. For more on counting augmentations, see for instance \cite{NgSab,HenryRutherford,NRSS}.
\end{remark}

Although augmentations and the augmentation category $\Aug_+$ are algebraic in nature, we will see in Section~\ref{sec:fillings} below that they can be modeled on a geometric source, exact Lagrangian fillings. Fillings of $\Lambda$ yield augmentations, and isotopic fillings induce isomorphic augmentations; see Theorem~\ref{thm:ehk} below. The entire augmentation category $\Aug_+(\Lambda,\kk)$ can be loosely viewed as an algebraic manifestation of an ``infinitesimally wrapped'' Fukaya category associated to $\Lambda$, whose objects are exact Lagrangian fillings of $\Lambda$ and whose morphisms are given by Floer homology groups $HF^*$ for a suitable perturbation of the Lagrangians. We refer the interested reader to \cite{NRSSZ} for further discussion of this viewpoint.

We close this section by mentioning a surprising connection between $\Aug_+$ and sheaf theory. Using techniques from algebraic geometry and inspired by work on microlocalization by Nadler--Zaslow \cite{NZ} and Guillermou--Kashiwara--Schapira, Shende, Treumann, and Zaslow \cite{STZ} defined $A_\infty$ categories $\Sh_n(\Lambda,\kk)$ associated to Legendrian knots $\Lambda$ in $\R^3$ or $ST^*\R^2$. The objects of $\Sh_n(\Lambda,\kk)$ are rank $n$ microlocal sheaves on $\R^2$ with microsupport on $\Lambda$, and the morphisms are given by $\Ext$ groups; see \cite{STZ} for the full definition. It was (essentially) conjectured in \cite{STZ}, and subsequently proven in \cite{NRSSZ}, that the augmentation and sheaf categories are equivalent for Legendrian $\Lambda$ in $\R^3$:
\[
\Aug_+(\Lambda,\kk) \cong \Sh_1(\Lambda,\kk).
\]

It is natural to ask what the augmentation analogue of $\Sh_n$ is for $n>1$. In fact, for any $n \geq 1$ one can assemble the set of $n$-dimensional representations of the DGA of $\Lambda$, as discussed in Section~\ref{ssec:rep}, into the objects of an $A_\infty$ category $\Rep_n(\Lambda,\kk)$; for $n=1$, we have $\Rep_1(\Lambda,\kk) = \Aug_+(\Lambda,\kk)$. It is conjectured that $\Rep_n(\Lambda,\kk) \cong \Sh_n(\Lambda,\kk)$ for all $n$. See \cite{CNS-rep} for the definition of $\Rep_n$ and some evidence for this conjecture.

% ------------------------------------------------------------------------------------------------------
\section{Fillings and Augmentations}\label{sec:fillings}
In this section we consider Lagrangian cobordisms between Legendrian knots and discuss how they induce maps on the Chekanov--Eliashberg DGA of the Legendrian knots. We then see how these maps can be used to obstruct cobordism and Lagrangian fillings of Legendrian knots. In particular, we will discuss connections between Lagrangian fillings of Legendrian knots and augmentations of the Chekanov--Eliashberg DGA. 

\subsection{Cobordisms and functoriality}

One nice feature of LCH, as predicted by the framework of Symplectic Field Theory \cite{EGH}, is that it is functorial in a particular way. To state this precisely, we need the notion of an exact Lagrangian cobordism between two Legendrian knots or links: see Figure~\ref{fig:cobordism} for an illustration.

\begin{definition}
Let $\Lambda_+,\Lambda_-$ be Legendrian links in $\R^3$. A \textit{Lagrangian cobordism from $\Lambda_-$ to $\Lambda_+$} is a Lagrangian submanifold $L$ of the symplectization $(\R \times \R^3,d(e^t\alpha))$ such that for some $T>0$, $L \cap ((-\infty,-T) \times \R^3) = (-\infty,-T) \times \Lambda_-$ and $L \cap ((T,\infty) \times \R^3) = (T,\infty) \times \Lambda_+$. 
\label{def:cobordism}
A Lagrangian cobordism $L$ is \textit{exact} if there is a function $f:\thinspace L\to\R$ such that $(e^t \alpha)|_L = df$ and $f$ is constant on each individual end, $(-\infty,-T) \times \Lambda_-$ and $(T,\infty) \times \Lambda_+$.
\end{definition}

\begin{figure}
{\tiny
\begin{overpic}{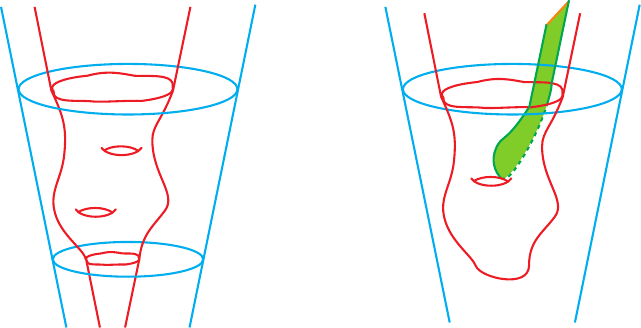}
\put(86, 115){$\Lambda_+$}
\put(275, 112){$\Lambda_+$}
\put(70, 30){$\Lambda_-$}
\put(80, 85){$L$}
\put(262, 40){$L$}
\put(115, 80){$\R\times \R^3$}
\put(240, 80){$\Delta$}
\put(265, 155){$a$}
\end{overpic}}
\caption{
On the left, a Lagrangian cobordism from $\Lambda_-$ to $\Lambda_+$. On the right, a disk contributing to the augmentation $\epsilon_L$ induced by a filling $L$ of $\Lambda_+$.}
\label{fig:cobordism}
\end{figure}

\begin{remark}
The notion of (exact) Lagrangian cobordism can be generalized by replacing $\R^3$ by an arbitrary contact $3$-manifold $(Y,\alpha)$, and further by replacing the symplectization $\R\times Y$ by an exact symplectic cobordism from $(Y,\alpha)$ to itself: that is, an exact symplectic manifold with two noncompact ends that agree with the symplectization $\R\times Y$. The functoriality of LCH extends to these more general circumstances; see \cite{EHK}.
\end{remark}

One can construct a ``cobordism category'' whose objects are Legendrian links in $\R^3$ and whose morphisms are exact Lagrangian cobordisms. (The condition in Definition~\ref{def:cobordism} that $f$ is constant on the ends ensures that exact Lagrangian cobordisms can be composed by concatenation, see \cite{Chan-disc}.) The following result, roughly speaking, says that LCH gives a contravariant functor from this cobordism category to the category of DGAs.

\begin{theorem}[Ekholm--Honda--K\'alm\'an \cite{EHK}]
An exact Lagrangian cobordism $L$ 
\label{thm:ehk}
from $\Lambda_-$ to $\Lambda_+$ induces a DGA map between Chekanov--Eliashberg DGAs
\[
\Phi_L :\thinspace (\A_{\Lambda_+},\d_{\Lambda_+}) \to (\A_{\Lambda_-},\d_{\Lambda_-}),
\]
that is, an algebra map $\Phi_L :\thinspace \A_{\Lambda_+} \to \A_{\Lambda_-}$ such that $\Phi_L \circ \d_{\Lambda_+} = \d_{\Lambda_-} \circ \Phi_L$.
The maps $\Phi_L$ satisfy the following properties:
\begin{enumerate}
\item
if $L = \R\times\Lambda$ is a trivial Lagrangian cylinder, then $\Phi_L = \id_{\A_\Lambda}$;
\item 
if $L_1,L_2$ have the same ends $\Lambda_\pm$ and are isotopic through exact Lagrangian cobordisms, then $\Phi_{L_1},\Phi_{L_2}$ are chain homotopic;
\label{it:ehk2}
\item
if $L_1,L_2$ are exact Lagrangian cobordisms from $\Lambda_0$ to $\Lambda_1$ and from $\Lambda_1$ to $\Lambda_2$, respectively, and $L$ is the cobordism from $\Lambda_0$ to $\Lambda_2$ obtained by concatenating $L_1$ and $L_2$, then $\Phi_L$ is chain homotopic to $\Phi_{L_1} \circ \Phi_{L_2}$.
\end{enumerate}
\end{theorem}

There are some subtleties in the precise content of Theorem~\ref{thm:ehk} that we discuss in the following two remarks.

\begin{remark}[coefficients]
Theorem~\ref{thm:ehk} is stated in \cite{EHK} with the DGAs being over $\Z_2$ and with no homology coefficients. 
\label{rmk:coeffs}
One can readily lift Theorem~\ref{thm:ehk} to include homology coefficients by choosing base points on $\Lambda_\pm$ and paths on $L$ connecting these base points; see e.g. \cite{CNS,Pan-fillings}. To lift Theorem~\ref{thm:ehk} from $\Z_2$ to $\Z$, one needs to coherently orient the moduli spaces that are used in the proof. This can be done in general when $L$ is spin, and in particular for $\dim L = 2$ when $L$ is orientable; see \cite{Karlsson-orient}.
\end{remark}

\begin{remark}[gradings]
If either of $\Lambda_\pm$ is a disconnected link, then the grading on the corresponding DGA is not well-defined and relies on a collection of choices; see Remark~\ref{rmk:links}. 
\label{rmk:gradings}
For the map $\Phi_L$ 
to preserve grading, the choices for $\Lambda_\pm$ need to be compatible in a suitable sense involving $L$. 

Even when both of $\Lambda_\pm$ are single-component knots, the extent to which $\Phi_L$ preserves grading depends on the Maslov number $m(L)$ of the Lagrangian $L$, defined to be the $\gcd$ of the Maslov numbers of all closed loops in $L$. If $m(L) = 0$ then $\Phi_L$ preserves the full $\Z$ grading on the DGAs; otherwise it preserves only the induced quotient grading in $\Z/(m(L)\Z)$. In particular, an oriented cobordism preserves at least a $\Z_2$ grading, while an unoriented cobordism need not preserve any grading at all. See \cite{EHK} for some discussion of these grading issues.
\end{remark}

\subsection{Decomposable cobordisms}

Here we briefly discuss how to construct exact Lagrangian cobordisms between Legendrian links, following \cite{EHK}. Say that a crossing in the Lagrangian projection of a Legendrian link is \textit{contractible} if the height of the corresponding Reeb chord can be made arbitrarily close to $0$ by a Legendrian isotopy of the link that corresponds to a planar isotopy of the Lagrangian projection.

\begin{theorem}[\cite{EHK}]
Let $\Lambda_\pm$ be Legendrian links in $\R^3$. 
\label{thm:decomp}
There is an exact Lagrangian cobordism from $\Lambda_-$ to $\Lambda_+$ if $\Lambda_-$ is obtained from $\Lambda_+$ by one of the following:
\begin{itemize}
\item
Legendrian isotopy;
\item
deleting a component of $\Lambda_+$ that is a standard Legendrian unknot (with $\tb=-1$) and is contractible in the complement of the remainder of $\Lambda_+$ (``unknot filling'');
\item
the ``pinch move'' shown in Figure~\ref{fig:decomp}, which is a saddle move in the $xz$ projection and a ``0-resolution'' of a contractible crossing in the $xy$ projection.
\end{itemize}
\end{theorem}

\begin{figure}[ht]{\tiny
\begin{overpic}{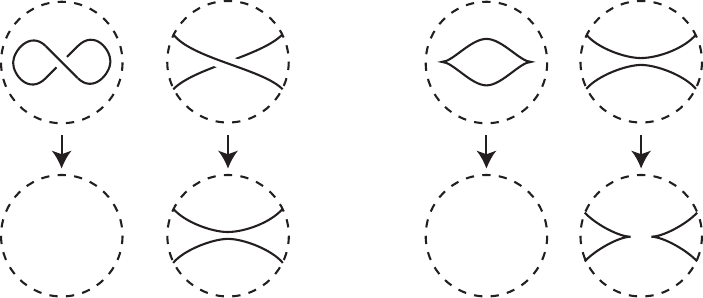}
\end{overpic}}
\caption{Two local moves representing exact Lagrangian cobordisms:
the left two figures in each group are unknot filling; the right two figures are a pinch move. The left most group shows the Lagrangian projections and the right most group shows the front projections.  
(Arrows indicate going from the top of the cobordism towards the bottom.) 
For the pinch move in the Lagrangian projection, it is crucial that the crossing being resolved is contractible.
}
\label{fig:decomp}
\end{figure}

Any concatenation of the ``elementary'' cobordisms listed in Theorem~\ref{thm:decomp} yields an exact Lagrangian cobordism, which we call \textit{decomposable}. It is currently an open question whether any exact Lagrangian cobordism is (Lagrangian isotopic to) a decomposable cobordism.

\begin{example}
In practice, one constructs decomposable cobordisms from top to bottom, starting with $\Lambda_+$ and successively applying pinch moves and unknot fillings. 
\label{ex:trefoil-fillings}
An illustrative example from \cite{EHK} is when $\Lambda_+$ is the standard Legendrian right-handed trefoil, shown in Figure~\ref{fig:trefoil}. The crossings $a_3, a_4, a_5$ are all contractible: each of them can be made to have arbitrarily small height. One can construct five decomposable cobordisms from the empty set to $\Lambda_+$ (``fillings'', see Section~\ref{ssec:fillings} below) as follows.
Let $i,j$ be two distinct integers in $\{1,2,3\}$. Apply a pinch move to the crossing in the $xy$ projection of $\Lambda_+$ labeled by $i$, followed by a pinch move to the crossing labeled by $j$. The result is a standard Legendrian unknot, which we can then delete by unknot filling, resulting in the empty set. Of the six decomposable cobordisms corresponding to different choices of $(i,j)$, it can be shown that $(i,j) = (1,3)$ and $(3,1)$ yield isotopic cobordisms. It is proven in \cite{EHK} that the remaining five cobordisms are pairwise non-isotopic; see Section~\ref{ssec:fillings}.
\end{example}

\subsection{Fillings}
\label{ssec:fillings}

We now focus on a particular case of an exact Lagrangian cobordism, when the negative end $\Lambda_-$ is empty. In this case the cobordism is called an \textit{exact Lagrangian filling} of the Legendrian link $\Lambda_+$. With the simplest possible choice of coefficients as in \cite{EHK}, the DGA of the empty link is $(\Z_2,0)$, and it follows from Theorem~\ref{thm:ehk} that an exact Lagrangian filling of $\Lambda_+$ gives an augmentation from $(\A_{\Lambda_+},\d_{\Lambda_+})$ to $(\Z_2,0)$. When $L$ is orientable, from \cite{Karlsson-orient} we can lift the augmentation from $\Z_2$ to $\Z$. Indeed, we have the following result.

\begin{theorem}
Let $L$ be a connected, orientable, exact Lagrangian filling of a Legendrian link $\Lambda$. 
\label{thm:filling-aug}
Then $L$ induces an augmentation
\[
\e_L :\thinspace (\A_\Lambda,\d_\Lambda) \to (\Z[\pi_1(L)],0).
\]
\end{theorem}

Here we sketch the definition of the map $\e_L$, following the more general construction of cobordism maps in \cite{EHK}. Let $a$ be a Reeb chord of $\Lambda$, and let $\M(a)$ denote the moduli space of rigid holomorphic disks in $\R\times\R^3$ with boundary on $L$ and a single positive puncture asymptotic to $a$. For $\Delta \in \M(a)$, we can concatenate the oriented boundary $\partial \Delta$ with the capping path for $a$ in $\Lambda$ to produce an element $[\Delta] \in \pi_1(L)$. Now define
\[
\e_L = \sum_{\Delta \in \M(a)} (\sgn(\Delta)) [\Delta]
\]
where $\sgn(\Delta) \in \{\pm 1\}$ is a sign coming from the orientation of $\M(a)$, and extend $\e_L$ to an algebra map on all of $\A_{\Lambda}$. See Figure~\ref{fig:cobordism} for an illustration.

A pictorial sketch of the proof of Theorem~\ref{thm:filling-aug} is given in Figure~\ref{fig:filling-aug}. Briefly, one follows standard Floer-type arguments by considering the compactification of $\M^1(a)$, the $1$-dimensional moduli space of holomorphic disks in $\R\times\R^3$ with boundary on $L$ and a positive puncture at $a$. Contributions to the boundary of $\M^1(a)$ come from a holomorphic disk in $(\R\times\R^3,\R\times\Lambda)$ with positive puncture at $a$ and some number of negative punctures, glued to holomorphic disks in $(\R\times\R^3,L)$. Each of these contributions counts a term in $\e_L(\d(a))$; for instance, in the left diagram in Figure~\ref{fig:filling-aug}, we have (disregarding elements of $\pi_1(L)$) $\d(a) = a_1a_2+\cdots$ and $\e_L(\d(a)) = \e_L(a_1)\e_L(a_2)+\cdots$. Since the compactification of $\M^1(a)$ is a compact $1$-manifold, these terms must cancel in pairs, yielding the theorem. 
It should be noted that the exactness of $L$ rules out one possibly problematic degeneration in $\M^1(a)$, ``boundary disk bubbling'', as shown in the right diagram in Figure~\ref{fig:filling-aug}: there can be no nontrivial holomorphic disk $\Delta$ with boundary fully on $L$, because the area of $\Delta$ would be $\int_\Delta \omega = \int_{\partial \Delta} e^t\alpha = 0$ by exactness.

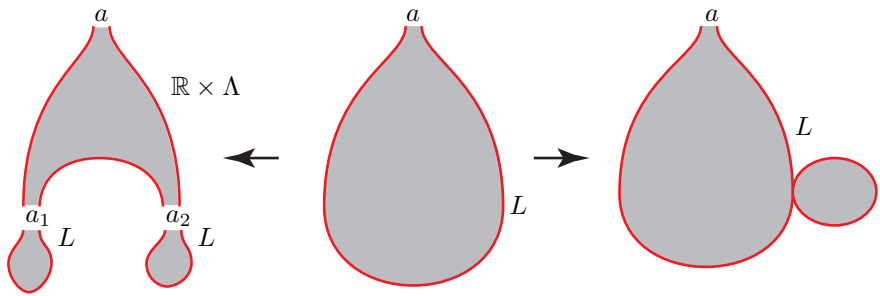
\begin{figure}
{\small
\begin{overpic}{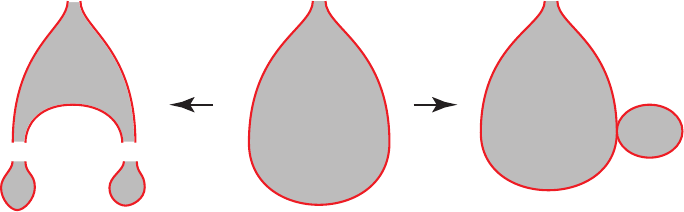}
\put(62, 75){$\R\times \Lambda$}
\put(19, 18){$L$}
\put(72, 18){$L$}
\put(190, 30){$L$}
\put(298, 60){$L$}
\put(33, 103){$a$}
\put(7, 27){$a_1$}
\put(60, 27){$a_2$}
\put(151, 103){$a$}
\put(264, 103){$a$}
\end{overpic}}
\caption{
Possible degenerations of disks in $\M^1(a)$ (the right one is actually forbidden).
}
\label{fig:filling-aug}
\end{figure}

\begin{example}
Consider the trefoil $\Lambda$ from Examples~\ref{ex:trefoil} and~\ref{ex:trefoil-fillings}.
Using a combinatorial formula for the cobordism maps corresponding to pinch moves, it is computed in \cite{EHK} that the five fillings of the $\Lambda$ from produce the five distinct augmentations $(\A_\Lambda,\d_\Lambda) \to (\Z_2,0)$ (where we quotient the augmentations from Theorem~\ref{thm:filling-aug} by the map $\Z[\pi_1(L)] \to \Z_2$). For grading reasons, these augmentations are not chain homotopic to each other, and it follows from Theorem~\ref{thm:ehk} that the five fillings are all non-isotopic. In \cite{Pan-fillings}, Pan generalizes this result of Ekholm--Honda--K\'alm\'an to produce $\frac{1}{n+1}\binom{2n}{n}$ distinct fillings of the Legendrian $(2,n)$ torus knot for $n \geq 1$; in the general case, not all of these fillings induce distinct augmentations to $\Z_2$, but they do induce all distinct augmentations to $\Z_2[\pi_1(L)]$. 
\end{example}

\begin{remark}
When constructing Fukaya categories, one often considers not exact Lagrangians but exact Lagrangians equipped with local systems. In our context, a rank $n$ local system on an exact filling $L$ of a Legendrian knot $\Lambda$ consists of a representation $\pi_1(L) \to GL(n,\kk)$ for some $n$ and some field $\kk$. If we compose this representation with the ``universal'' augmentation given in Theorem~\ref{thm:filling-aug}, we obtain a DGA map $(\A_\Lambda,\d_\Lambda) \to (\Mat_n(\kk),0)$. That is,  in the terminology of Section~\ref{ssec:rep}, an exact filling of $\Lambda$ with a rank $n$ local system induces an $n$-dimensional representation of $(\A_\Lambda,\d_\Lambda)$.
\end{remark}

\subsection{Augmentations not from fillings}

From the preceding discussion, any exact Lagrangian filling of a Legendrian knot $\Lambda$ induces an augmentation of the DGA $(\A_\Lambda,\d_\Lambda)$, to say $\Z_2$ for simplicity. It is however not the case that all augmentations come from fillings.
As an example, consider the Legendrian figure eight knot $\Lambda$ shown in the left of Figure~\ref{fig:figure-eight}. It is readily checked that the DGA for $\Lambda$ has a unique augmentation to $\Z_2$. However, there is a topological obstruction to $\Lambda$ having an (embedded, orientable) Lagrangian filling, exact or not. If $L$ were such a filling, then by work of Chantraine \cite{chan}, we would have $\tb(\Lambda) = 2g(L)-1$, where $g(L)$ is the genus of $L$; but $\tb(\Lambda)=-3$.

\begin{figure}
{\tiny
\begin{overpic}{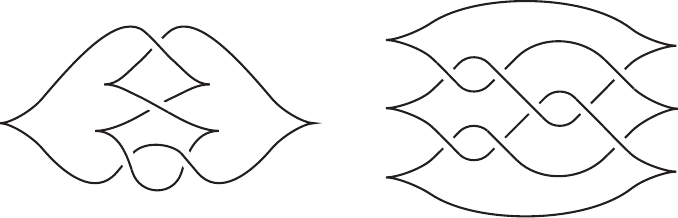}
\end{overpic}}
\caption{
Two Legendrian knots with augmentations that do not come from fillings: the figure eight (left) and a knot of type $m(8_{21})$ (right).
}
\label{fig:figure-eight}
\end{figure}

A subtler obstruction to augmentations coming from fillings is provided by the so-called Seidel isomorphism. This relates the homology of a filling to the linearized LCH of the corresponding augmentation, and was for a while a folk result in the subject derived from an observation of Seidel (see \cite{EkholmSFT} for a statement from the work of Ekholm) before being formally proven by Dimitroglou Rizell \cite{DR}.

\begin{theorem}[Seidel isomorphism]\label{si}
Let $L$ be an exact Lagrangian filling of a Legendrian knot $\Lambda$, 
\label{thm:seidel}
and let $\e_L :\thinspace (\A_\Lambda,\d_\Lambda) \to (\Z_2,0)$ be the induced augmentation to $\Z_2$. Then
\[
\LCH^\ast_{\e_L}(\Lambda) \cong H_{2-\ast}(L;\Z_2).
\]
\end{theorem}

\noindent
For example, for the trefoil $\Lambda$ with filling $L$ (topologically a punctured torus), $\LCH^*_\e(\Lambda)$ was computed in Example~\ref{ex:trefoil-lin}, and it agrees with $H_{2-\ast}(L;\Z_2)$.

The proof in \cite{DR} of Theorem~\ref{thm:seidel} constructs an exact triangle relating the linearized LCH cochain complex of $\Lambda$, the Morse complex of $L$, and a wrapped Floer complex associated to $L$, and then observes that the wrapped Floer homology of $L$ vanishes. Theorem~\ref{thm:seidel} has been subsequently generalized in several directions, notably to bilinearized LCH by Bourgeois and Chantraine \cite{BC}; this fits in with a larger picture of Floer homology associated to Lagrangian cobordisms, as developed by Chantraine, Dimitroglou Rizell, Ghiggini, and Golovko \cite{CDGG1,CDGG2}.

\begin{example}
Consider the Legendrian $m(8_{21})$ knot shown on the right of Figure~\ref{fig:figure-eight}. This knot was famously considered by Melvin and Shrestha \cite{MS} and has the unusual property that it has augmentations with different linearized LCH. For one set of augmentations, the Poincar\'e polynomial for $\LCH^*_\e$ is $t^2+2t$, while for another set it is $2t^2+4t+1$. The first set can (and indeed does) come from oriented exact Lagrangian fillings. We claim that the second set cannot, because of the Seidel isomorphism. Indeed, any oriented filling must have even Maslov number, whence the Seidel isomorphism holds at least for grading mod $2$. Any oriented exact filling $L$ must be connected (by Stokes, there are no closed exact Lagrangian surfaces in $\R\times\R^3$) and thus satisfies $H_{\text{even}}(L;\Z_2) \cong \Z_2$, while $\LCH^{\text{even}}_\e \cong (\Z_2)^3$.
\end{example}

Theorem~\ref{si} is very useful at obstructing a Legendrian knot from having an exact Lagrangian filling. For example in \cite{LipSab} Lipman and Sabloff use this result to completely characterize which Legendrian ``4-plat knots'' have fillings. In the opposite direction, 
Etg\"u \cite{Etgu} has shown that there are Legendrian knots with augmentations whose linearized contact homology is isomorphic to the homology of a surface, in accordance with Theorem~\ref{si}, but which do not come from any filling.

\begin{remark}[Coefficients and gradings]
As in Remark~\ref{rmk:coeffs}, the work of Karlsson \cite{Karlsson-orient} can be used to promote the Seidel isomorphism to arbitrary coefficients in the case where the filling $L$ is orientable; for some discussion, see \cite{CDGG2}. As in Remark~\ref{rmk:gradings}, the extent to which the Seidel isomorphism is graded depends on the Maslov number of $L$. For instance, in the setting where $L$ is orientable but does not necessarily have Maslov number $0$, the isomorphism is only guaranteed to hold when the gradings are taken in $\Z_2$.
\end{remark}

\begin{remark}[Interpretation in $\Aug_+$]
It is observed in \cite{NRSSZ} that the Seidel isomorphism can be reinterpreted in a natural way in the augmentation category $\Aug_+(\Lambda,\kk)$. Here the statement of Theorem~\ref{si} becomes: if $L$ is an exact Lagrangian filling of $\Lambda$ with augmentation $\e_L$, then
\[
H^*\Hom(\e_L,\e_L) \cong H^*(L).
\]
This version of the isomorphism bears a strong similarity to a foundational property in Lagrangian intersection Floer homology, where (roughly speaking) if $L$ is a Lagrangian then we have $HF^*(L,L) \cong H^*(L)$. This is in accordance with the interpretation of $\Aug_+$ as a version of a Fukaya category, as discussed previously in Section~\ref{ssec:aug-properties}.
\end{remark}

% ------------------------------------------------------------------------------------------------------
\section{LCH and Weinstein Domains}\label{LCHandWeinstein}

So far, we have tried to provide a self-contained introduction to Legendrian contact homology and the Chekanov--Eliashberg DGA, viewed as interesting invariants of Legendrian knots. However, Legendrian contact homology also occupies a key role in modern symplectic topology through its role in studying Liouville and Weinstein domains. In this section we give a very limited and rather sketchy discussion of this picture; more details can be found in the references. The reader is cautioned that this story is currently rapidly developing and parts of it are not entirely rigorous at the moment.

The beginning point for this discussion is a certain type of symplectic manifold with contact boundary called a \textit{Liouville domain} \cite{Seidel-biased}. This is a compact symplectic manifold $(X,\omega)$ such that $\omega = d\lambda$ is exact with primitive $1$-form $\lambda$, resulting in the \textit{Liouville vector field} $Z$ on $X$ defined by $\lambda = i_Z\omega$, and such that $Z$ points outwards along $\partial X$. The boundary $Y = \partial X$ is then a contact manifold with contact $1$-form $\lambda$, and near the boundary $X$ looks like the symplectization of $Y$.

A Liouville domain $X$ is called a \textit{Weinstein domain} \cite{EG} if it is equipped with a Morse function $\phi$ that is locally constant on $\partial X$ and for which the Liouville vector field $Z$ is gradient-like. A nice feature of Weinstein domains is that one can adapt the standard topological handle-decomposition picture for $X$ from the Morse theory of $\phi$ to the symplectic setting. If $\dim X = 2n$, then each handle in the handle decomposition is of index $\leq n$, and each one is modelled by a standard symplectic handle called a \textit{Weinstein handle}. The handles of index $<n$ and $n$ are called subcritical and critical, respectively. One can then build up $X$ by first attaching all of the subcritical handles, resulting in a ``subcritical Weinstein domain'' $X_0$, and then attaching the critical handles. These critical handles are attached to $X_0$ along attaching spheres in the contact boundary $\partial X_0$ which are in fact Legendrian. The symplectic topology of the subcritical domain $X_0$ turns out to be fairly simple, and the interesting symplectic topology of $X$ is determined by the Legendrian attaching spheres in $\partial X_0$.

This leads to the following picture. Let $X_0$ be a subcritical Weinstein domain with contact boundary $Y_0$, and let $\Lambda$ be a Legendrian sphere in $Y_0$. We can then construct a Weinstein domain $X$ by attaching a Weinstein handle to $X_0$ along $\Lambda$; the isotopy type of $\Lambda$ determines $X$ up to symplectomorphism, and the boundary $\partial X$ is obtained from $\partial X_0$ by Legendrian surgery on $\Lambda$. See Figure~\ref{fig-Weinstein} for a schematic picture.

\begin{figure}
{\small
\begin{overpic}{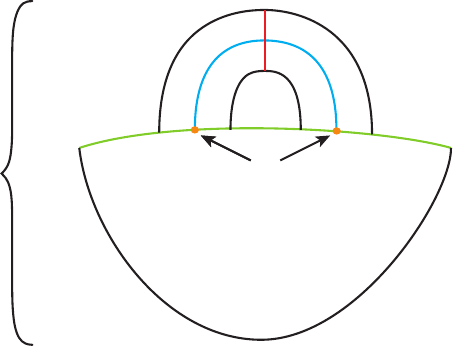}
\put(120, 40){$X_0$}
\put(-9, 80){$X$}
\put(124, 84){$\Lambda$}
\put(159, 130){$\Delta_\Lambda$}
\put(130, 150){$C_\Lambda$}
\put(200, 102){$Y_0$}
\end{overpic}}
\caption{Attaching a critical Weinstein handle to a subcritical Weinstein domain $X_0$ along a Legendrian sphere $\Lambda \subset Y_0 = \partial X_0$ to produce a Weinstein domain $X$. Also pictured are the Lagrangian core $\Delta_\Lambda$ and cocore $C_\Lambda$ of the handle.}
\label{fig-Weinstein}
\end{figure}
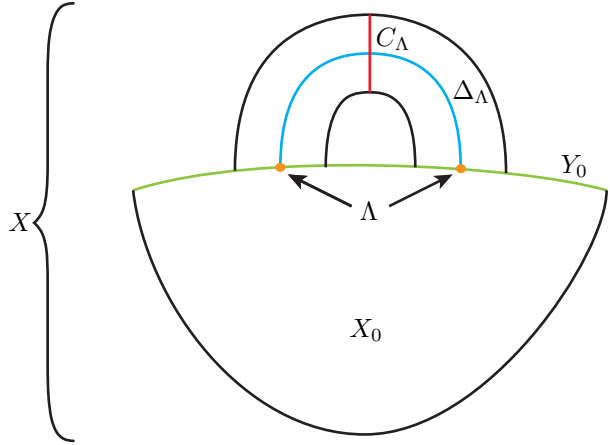

There are various interesting symplectic invariants that one can associate to $X$. Key among these are \textit{linearized contact homology} $CH_*(X)$, which is the contact homology of the boundary $\partial X$ linearized by the augmentation coming from the filling $X$, and \textit{symplectic homology} $SH_*(X)$. See e.g. \cite{BEE} for definitions and a history of these invariants.

A key result announced by Bourgeois, Ekholm, and Eliashberg \cite{BEE} is that both $CH_*(X)$ and $SH_*(X)$ are essentially determined by the Chekanov--Eliashberg DGA $(\A_\Lambda,\d_\Lambda)$ of $\Lambda$. For linearized contact homology, there is an exact triangle
\[
\cdots \to CH(X) \to CH(X_0) \to LCH^{\text{cyc}}(\Lambda) \to \cdots
\]
where $LCH^{\text{cyc}}(\Lambda)$ is the \textit{cyclic Legendrian contact homology} of $\Lambda$: the homology of the complex generated by cyclic words in Reeb chords of $\Lambda$, with differential induced by $\d_\Lambda$. For symplectic homology, one can define another homology $LCH^{\text{Ho}}_*(\Lambda)$ derived from $(\A_\Lambda,\d_\Lambda)$ using a construction analogous to Hochschild homology; the precise definition of $LCH^{\text{Ho}}_*(\Lambda)$ is a bit involved and we refer the reader to \cite{BEE}. We then have the following result.

\begin{theorem}[{\cite[Corollary 5.7]{BEE}}]
There is an isomorphism $SH_*(X) \cong LCH^{\text{Ho}}_*(\Lambda).$
\end{theorem}

We note that the proofs of the results announced in \cite{BEE}, including the above results about $CH_*(X)$ and $SH_*(X)$, have not yet appeared. Nevertheless, the main takeaway is that both linearized contact homology $CH_*(X)$ and symplectic homology $SH_*(X)$ are determined by the DGA $(\A_\Lambda,\d_\Lambda)$ of the Legendrian attaching sphere $\Lambda$.

In the special case where $\dim X = 4$, the subcritical domain $X_0$ can be decomposed into $0$-handles and $1$-handles, and the boundary $\partial X_0$ is a connected sum $\#^k(S^1 \times S^2)$. The DGA of a Legendrian knot or link in $\#^k(S^1 \times S^2)$ has been combinatorially described in \cite{EN}, generalizing the $k=0$ case, which corresponds to the contact manifold $S^3$ and where it can be shown that the DGA is the same as the one we have considered for Legendrian knots in $\R^3$. It follows that for any Weinstein domain $X$ of dimension $4$, there is a combinatorial description for $CH_*(X)$ and $SH_*(X)$ in terms of a diagram for the Legendrian knot or link in $\#^k(S^1 \times S^2)$ along which the critical handles are attached. As one sample consequence, it can be shown using $CH_*$ that the contact $3$-manifolds obtained from $S^3$ by Legendrian surgery on the Chekanov $m(5_2)$ knots, while the same as smooth manifolds, are distinct as contact manifolds; see \cite{BEE}.

A more direct interpretation of LCH as it relates to Weinstein domains is given by wrapped Floer homology. To set this up, we use the same setup as before: let $X_0$ be a Weinstein (or Liouville) domain, let $\Lambda$ be a Legendrian sphere in the contact boundary $\partial X_0$, and let $X$ be the Liouville domain obtained from $X_0$ by attaching a Weinstein handle along $\Lambda$. The core of the handle is a Lagrangian disk $\Delta_\Lambda$ and the handle itself is then symplectomorphic to $T^*\Delta$. A fiber of this cotangent bundle is another Lagrangian disk, the \textit{cocore disk} $C_\Lambda$, which intersects $\Delta_\Lambda$ once and whose boundary lies on $\partial X$. See Figure~\ref{fig-Weinstein}.

To the Lagrangian cocore $C_\Lambda$ one can associate an invariant called the \textit{wrapped Floer homology} $HW_*(C_\Lambda)$. The following result has been announced in \cite{BEE}, with a proof sketch given in \cite{EL}:

\begin{theorem}
There is an isomorphism between $HW_*(C_\Lambda)$ and the full Legendrian contact homology $LCH_*(\Lambda) = H_*(\A_\Lambda,\d_\Lambda)$.
\end{theorem}

\noindent
One can interpret this result on the level of categories. The cocore $C_\Lambda$ is an object in the wrapped Fukaya category of $X$ and indeed generates this wrapped category \cite{CDGG-generation}. The endomorphism algebra of the full subcategory corresponding to the single object $C_\Lambda$ is then $A_\infty$ quasi-isomorphic to the DGA $(\A_\Lambda,\d_\Lambda)$. See \cite[Theorem~2]{EL}.

There is (conjecturally) a similar interpretation of the loop space DGA (see Remark~\ref{rmk:dga-flavors}) in terms of a partially wrapped version of Floer homology \cite{Sylvan}, cf.\ \cite[Theorem~2]{EL}, and this fits into a broader picture of Ganatra, Pardon, and Shende concerning partially wrapped Fukaya categories and Liouville sectors. See
\cite{GPS3} for further results in this direction.

% ------------------------------------------------------------------------------------------------------
\appendix

\section{The DGA of the Pretzel Knot $P(3,-3,-4)$}\label{app:unknotDGA}

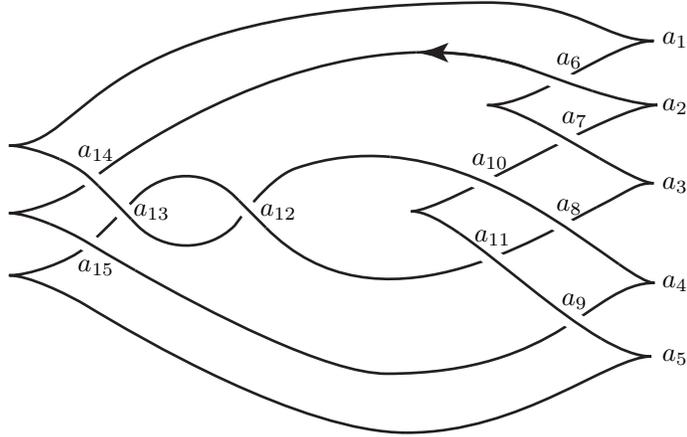
\begin{figure}[b]
{\small
\begin{overpic}{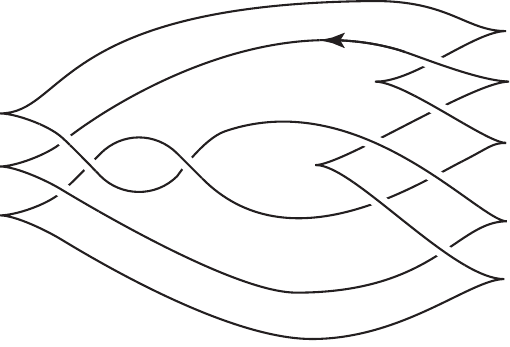}
\put(247, 151){$a_1$}
\put(247, 126){$a_2$}
\put(247, 96){$a_3$}
\put(247, 59){$a_4$}
\put(247, 31){$a_5$}
\put(207, 143){$a_6$}
\put(209, 121){$a_7$}
\put(207, 88){$a_8$}
\put(209, 52){$a_9$}
\put(175, 105){$a_{10}$}
\put(176, 76){$a_{11}$}
\put(95, 86){$a_{12}$}
\put(47, 86){$a_{13}$}
\put(26, 108){$a_{14}$}
\put(26, 65){$a_{15}$}
\end{overpic}}
\caption{A Legendrian knot of type $P(3,-3,-4)=m(10_{140})$. Crossings and right cusps (corresponding to Reeb chords for the resolution of this front) are labeled. A base point is placed in the loop at $a_5$ produced by the resolution.}
\label{fig:m10140}
\end{figure}

Here we prove Proposition~\ref{prop:unknotDGA} for the case $m=1$ by providing an explicit stable tame isomorphism between the DGA for the Legendrian pretzel knot $P(3,-3,-4)$ shown in Figure~\ref{fig:m10140}, which we call $\Lambda$, and the DGA for the unknot from Example~\ref{ex:unknot}.
The knot $\Lambda$ has $15$ Reeb chords, of the following degrees:
\begin{align*}
2: & \quad  a_8, a_{13} \\
1: & \quad  a_1, a_2, a_3, a_4, a_5,a_{10}, a_{15} \\
0: & \quad  a_6, a_7, a_{11}, a_{14} \\
-1: & \quad  a_9 \\
-2: & \quad  a_{12}.
\end{align*}

The DGA $(\A_\Lambda,\d=\d_\Lambda)$ is generated by $a_1,\ldots,a_{15}$, along with $t^{\pm 1}$ in degree $0$. The differential is given as follows:
\begin{align*}
\d(a_1) &= 1+a_{14}a_6 \\
\d(a_2) &= 1-a_6a_7+a_{15}a_{12}a_{10} \\
\d(a_3) &= 1 - a_7 a_{11} \\
\d(a_4) &= 1+a_{10}a_9-a_{14}-a_8a_{12}a_{14} \\
\d(a_5) &= t^{-1}-a_{11}-a_{11}a_{12}a_{13}+a_9a_{15} \\
\d(a_8) &= a_{10}a_{11} \\
\d(a_9) &= -a_{11}a_{12}a_{14} \\
\d(a_{13}) &= -a_{14}a_{15}
\end{align*}
and $\d(a_i) = 0$ for all other $i \leq 15$.

\begin{remark}
Before we present the stable tame isomorphism between this DGA and the DGA for the unknot, we comment on the motivation for this computation, which comes from the characteristic algebra \cite{Ng-CLI}. The characteristic algebra $\mathcal{C}$ of $(\A_\Lambda,\d)$ is defined to be the quotient of $\A_\Lambda$ by the two-sided ideal generated by $\{\d(a_i)\}$, and is generally easier to handle than the homology of $(\A_\Lambda,\d)$ while still being invariant in a suitable sense (see \cite{Ng-CLI}). Here $\mathcal{C}$ is generated by $a_1,\ldots,a_{15},t^{\pm 1}$, and in $\mathcal{C}$ we have the following relations
\[
a_{14}a_6=-1, \qquad a_7a_{11} = 1, \qquad a_{11}a_{12}a_{14} = 0, \qquad a_6a_7 = 1+a_{15}a_{12}a_{10}
\]
from $\d(a_1)$, $\d(a_3)$, $\d(a_9)$, $\d(a_2)$ respectively. It follows that in $\mathcal{C}$,
$a_{12} = -a_7a_{11}a_{12}a_{14}a_6 = 0$ and so $a_6a_7 = 1$. Together with $a_7a_{11}=1$, this implies that $a_{11}=a_6a_7a_{11}=a_6$ and so
$a_6=a_{11}$ and $a_7$ are two-sided inverses of each other. We can successively use the rest of the relations in $\mathcal{C}$ coming from $\d(a_i)$ to conclude that $\C$ can be reduced to generators $a_1,a_2,a_3,a_4,a_5,a_8,a_9,a_{13},t^{\pm 1}$ with a single relation $1+t^{-1}$. This is equivalent to the characteristic algebra for the unknot, which has generators $a,t^{\pm 1}$ with the same single relation.
\end{remark}

We now proceed to the stable tame isomorphism between DGAs. 
In $(\A_\Lambda,\d)$, note that $\d(a_7a_9a_6) = -a_7a_{11}a_{12}a_{14}a_6$ and so
\[
\d(a_7a_9a_6-a_3a_{12}a_{14}a_6+a_{12}a_1) = a_{12}.
\]
Now stabilize $\A_\Lambda$ once by adding generators $a_{16},a_{17}$ with $|a_{16}|=0$, $|a_{17}|=-1$ and $\d(a_{16}) = a_{17}$, $\d(a_{17}) = 0$. Then if we conjugate by the elementary automorphism that sends
\[
a_{17} \mapsto a_{17}-(a_7a_9a_6-a_3a_{12}a_{14}a_6+a_{12}a_1)
\]
then the new differential, which we also write as $\d$, agrees with the original $\d$ except for $\d(a_{17})= a_{12}$ and $\d(a_{16}) = a_{17}-a_7a_9a_6+a_3a_{12}a_{14}a_6-a_{12}a_1$. We then use the following elementary automorphisms to remove $a_{12}$ from the differentials of all generators besides $a_{17}$:
\begin{align*}
a_2 &\mapsto a_2-a_{15}a_{17}a_{10} \\
a_5 &\mapsto a_5-a_{11}a_{17}a_{13} \\
a_9 &\mapsto a_9-a_{11}a_{17}a_{14}
\end{align*}
followed by
\begin{align*}
a_4 &\mapsto a_4 - a_8a_{17}a_{14} \\
a_{16} &\mapsto a_{16}-a_3a_{17}a_{14}a_6-a_{17}a_1.
\end{align*}
The end result is the following differential:
\begin{align*}
\d(a_1) &= 1+a_{14}a_6 \\
\d(a_2) &= 1-a_6a_7 \\
\d(a_3) &= 1 - a_7 a_{11} \\
\d(a_4) &= 1+a_{10}a_9-a_{14} \\
\d(a_5) &= t^{-1}-a_{11}+a_9a_{15} \\
\d(a_8) &= a_{10}a_{11} \\
\d(a_{13}) &= -a_{14}a_{15} \\
\d(a_{16}) &= -a_7a_9a_6\\
\d(a_{17}) &= a_{12}
\end{align*}
and $\d(a_i) = 0$ for all other $i\leq 17$.

Next note that $\d(-a_6a_{16}a_7+a_2a_9a_6a_7-a_9a_2) = a_9$. We stabilize $\A_\Lambda$ once more by adding generators $a_{18},a_{19}$ with $|a_{18}| = 1$, $|a_{19}| = 0$ and $\d(a_{18}) = a_{19}$, $\d(a_{19}) = 0$. Conjugate by the elementary automorphism
\[
a_{19} \mapsto a_{19} -(-a_6a_{16}a_7+a_2a_9a_6a_7-a_9a_2)
\]
to get $\d(a_{19}) = a_9$ and $\d(a_{18}) = a_{19} +a_6a_{16}a_7-a_2a_9a_6a_7+a_9a_2$. Now eliminate $a_9$ from the differentials of everything besides $a_{19}$ by applying
\begin{align*}
a_4 &\mapsto a_4-a_{10}a_{19} \\
a_5 &\mapsto a_5-a_{19}a_{15} \\
a_{16} &\mapsto a_{16}-a_7a_{19}a_6
\end{align*}
followed by
\[
a_{18} \mapsto a_{18}+a_{19}a_2+a_2a_{19}a_6a_7.
\]
The end result is:
\begin{align*}
\d(a_1) &= 1+a_{14}a_6 \\
\d(a_2) &= 1-a_6a_7 \\
\d(a_3) &= 1 - a_7 a_{11} \\
\d(a_4) &= 1-a_{14} \\
\d(a_5) &= t^{-1}-a_{11} \\
\d(a_8) &= a_{10}a_{11} \\
\d(a_{13}) &= -a_{14}a_{15} \\
\d(a_{17}) &= a_{12} \\
\d(a_{18}) &= a_6a_{16}a_7 \\
\d(a_{19}) &= a_9
\end{align*}
and $\d(a_i) = 0$ for all other $i\leq 19$.

It is now straightforward to reduce this DGA to the DGA of the unknot. Successively apply the following elementary automorphisms:
\begin{align*}
a_{14} &\mapsto a_{14}+1 \\
a_1 &\mapsto a_1-a_4a_6 \\
a_{13} &\mapsto a_{13}+a_4a_{15} \\
a_6 &\mapsto a_6-1 \\
a_2 &\mapsto a_2-a_1a_7 \\
a_7 &\mapsto a_7-1 \\
a_3 &\mapsto a_3-a_2a_{11} \\
a_{11} &\mapsto a_{11}-1 \\
a_5 &\mapsto a_5-a_3 \\
a_8 &\mapsto a_8-a_{10}a_3 \\
a_{18} &\mapsto a_{18}-a_1a_{16}+a_1a_{16}a_7-a_{16}a_2 
\end{align*}
to give
\begin{align*}
\d(a_1) &= a_6 & \d(a_8) &= -a_{10} \\
\d(a_2) &= a_7 & \d(a_{13}) &= -a_{15}\\
\d(a_3) &= a_{11} & \d(a_{17}) &= a_{12} \\
\d(a_4) &= -a_{14} & \d(a_{18}) &= a_{16} \\
\d(a_5) &= 1+t^{-1} & \d(a_{19}) &= a_9 \\
\end{align*}
and $\d(a_i)=0$ for all other $i \leq 19$. Destabilize by removing generators in pairs: $a_1,a_6$; $a_2,a_7$; $a_3,a_{11}$; $a_4,a_{14}$; $a_8,a_{10}$; $a_{13},a_{15}$; $a_{17},a_{12}$; $a_{18},a_{16}$; $a_{19},a_9$. This produces the DGA generated by $a_5$ alone, with differential $\d(a_5) = 1+t^{-1}$, and this is precisely the DGA of the unknot from Example~\ref{ex:unknot}.

\newcommand{\etalchar}[1]{$^{#1}$}

\newpage

\centerline{
\textbf{
Errata to the published version of this paper
}
}

%\vspace{6pt}

These are a list of known errata to the published version of ``Legendrian contact homology in $\mathbb{R}^3$'' (in \textit{Surveys in $3$-Manifold Topology and Geometry}). 
Page numbers refer to the arXiv version of this paper (1811.10966v4 or 1811.10966v5), and alternatively to the published version of the paper.

%\vspace{6pt}

\begin{itemize}
\item
p.\ 10 (published p.\ 112).
The sentence ``One may check that if $\Delta(a;b_1,\ldots,b_n)$ is nonempty then $|a|-\sum_{i=1}^n |b_i|=1$'' is correct if the rotation number of $\Lambda$ is $0$. However, when $\rot(\Lambda) \neq 0$, the statement is generally incorrect and should be amended as follows. Given any $u \in \Delta(a;b_1,\ldots,b_n)$, we can associate an integer $t(u)$ to be the number of times $\partial u$ passes through the base point $\ast$, counted with sign; in the language of p.\ 10 (published p.\ 112), we have $t(u) = \sum_{i=0}^n t(\eta_i)$. For $k\in\mathbb{Z}$, we then define
\[
\Delta_k(a;b_1,\ldots,b_n) = \{ u\in \Delta(a;b_1,\ldots,b_n)\,|\,t(u)=k\}.
\]
The correct statement is: if $\Delta_k(a;b_1,\ldots,b_n)$ is nonempty then
\[
|a|-\sum_{i=1}^n |b_i|+2k\rot(\Lambda) = 1.
\]
Note that this agrees with the differential $\partial_\Lambda$ having degree $-1$.
\item
p.\ 12 (published p.\ 114), Remark 3.4. The grading of $t_i$ should be $-2\rot(\Lambda_i)$ rather than $2\rot(\Lambda_i)$.
\item
p.\ 15 (published p.\ 117). In the definition of elementary automorphism, an elementary automorphism $\phi$ is not necessarily a chain map, but rather an algebra map that may change the differential. So $\phi$ should be viewed as an automorphism of $\mathbb{Z}\langle a_1,\ldots,a_n,t^{\pm 1}\rangle$ rather than of the DGA $(\mathbb{Z}\langle a_1,\ldots,a_n,t^{\pm 1}\rangle,\partial)$. It maps $(\mathbb{Z}\langle a_1,\ldots,a_n,t^{\pm 1}\rangle,\partial)$ to $(\mathbb{Z}\langle a_1,\ldots,a_n,t^{\pm 1}\rangle,\partial')$ where $\partial' = \phi \circ \partial \circ \phi^{-1}$.

Also, the expression $\pm t^k a_j t^\ell+u$ should be homogeneous in degree. Alternatively, we can restrict our notion of an elementary automorphism to send $\phi(a_j) = a_j + u$ where $|u|=|a_j|$ (so that every elementary automorphism is graded), and then enlarge the notion of tame isomorphism to include maps sending $a_i \mapsto \pm t^k a_i t_\ell$. This then includes tame isomorphisms that may not be graded, but the lack of grading-preserving is entirely due to ``changes of basis'' that replace a generator $a_i$ of the semifree algebra by another generator $\pm t^k a_i t_\ell$.
\item
p.\ 20 (published p.\ 112). The definition of $\epsilon(u)$ should be as follows: $\epsilon(u)$ is $\pm 1$ depending on whether the number of \textit{even-graded} $-$ corners in $u$ that cover a downward-facing quadrant is even or odd. (The phrase ``even-graded'' is missing from the definition in the paper.)
\end{itemize}

\end{document}